\documentclass[10pt,leqno]{article}

\usepackage{url}
\usepackage[pagebackref]{hyperref}
\usepackage{amsfonts}
\usepackage{amsthm}
\usepackage{amsmath}
\usepackage{amscd}
\usepackage{amssymb}
\usepackage{mathtools}
\usepackage{tikz}
\usepackage{tikz-cd}
\usepackage{color}
\usetikzlibrary{matrix,arrows}
\usepackage[T2A,T1]{fontenc}
\usepackage[utf8]{inputenc}
\usepackage[russian,english]{babel}

\def\CC {{\mathbb C}}     
\def\RR {{\mathbb R}}     
\def\ZZ {{\mathbb Z}}     

\def\ring#1{\ifmmode \mathaccent'027 #1\else \rm\accent'027 #1\fi}

\def\ol  {\overline}
\def\ul  {\underline}

\def\mc {\mathcal}

\def\tr {\mathrm{tr}}

\def \bd {\begin{diagram}}
\def \ed {\end{diagram}}
\def\be  {\begin{eqnarray}}
\def\ee  {\end{eqnarray}}
\def\ben {\begin{eqnarray*}}
\def\een {\end{eqnarray*}}

\def\bpr {\begin{proof}[Proof]}
\def\epr {\end{proof}}
\def\bsp {\begin{split}}
\def\esp {\end{split}}
\def\bcd {\begin{CD}}
\def\ecd {\end{CD}}

\mathchardef\mhyphen="2D

\newtheorem{theorem}{Theorem}[section]

\newtheorem*{thmdef2}{Theorem/Definition}

\newtheorem{prop}[theorem]{Proposition}
\newtheorem{coro}[theorem]{Corollary}
\theoremstyle{definition}
\newtheorem{remark}[theorem]{Remark}

\newtheorem{ex}[theorem]{Example}
\theoremstyle{plain}

\newtheorem*{theorem2}{Theorem}

\theoremstyle{plain} \newtheorem{thm}[theorem]{Theorem}
\theoremstyle{plain} \newtheorem{lem}[theorem]{Lemma}

\numberwithin{equation}{section}

\newcommand{\Addresses}{{
  \bigskip
  \footnotesize

	(F.~Haiden) \textsc{University of Oxford, Mathematical Institute, Andrew Wiles Building, Woodstock Road, Oxford OX2 6GG, UK} \par\nopagebreak
	\textit{E-mail:} \texttt{Fabian.Haiden@maths.ox.ac.uk}
	\medskip
	
	(L.~Katzarkov) \textsc{Fakult\"at f\"ur Mathematik, Universit\"at Wien, 
	Oskar-Morgenstern-Platz 1, 1090 Wien, Austria, HSE Moscow, and CMS Institute of Mathematics and Informatics, BAS Sofia, Bulgaria}\par\nopagebreak
	\textit{E-mail:} \texttt{lkatzarkov@gmail.com}
	\medskip
	
	(M.~Kontsevich) \textsc{Institut des Hautes \'Etudes Scientifiques,
	35 route de Chartres, 91440 Bures-sur-Yvette, France}\par\nopagebreak
	\textit{E-mail:} \texttt{maxim@ihes.fr}
	\medskip
	
	(P.~Pandit) \textsc{International Centre for Theoretical Sciences (ICTS-TIFR), Survey No. 151, Shivakote, Hesaraghatta Hobli, Bengaluru North 560089, India}\par\nopagebreak
	\textit{E-mail:} \texttt{pranav.pandit@icts.res.in}
	\medskip
}}

\begin{document}

\title{Semistability, modular lattices, and iterated logarithms}
\author{F.~Haiden, L.~Katzarkov, M.~Kontsevich, P.~Pandit}
\maketitle

\begin{abstract}
We provide a complete description of the asymptotics of the gradient flow on the space of metrics on any semistable quiver representation.
This involves a recursive construction of approximate solutions and the appearance of iterated logarithms and a limiting filtration of the representation. 
The filtration turns out to have an algebraic definition which makes sense in any finite length modular lattice.
This is part of a larger project by the authors to study iterated logarithms in the asymptotics of gradient flows, both in finite and infinite dimensional settings.
\end{abstract}

\tableofcontents

\section{Introduction}

This paper consists of two parts. 
The first is lattice-theoretic (lattice in the sense of partial order) and the main result is the existence of a weight-type filtration, depending on finitely many real parameters, in any finite-length modular lattice.
In the second part we study the asymptotic behavior of the gradient flow on the space of Hermitian metrics on a quiver representation, which involves iterated logarithms, i.e. the functions $\log t$, $\log\log t$, $\log\log\log t$, \ldots, and turns out to be controlled by the filtration defined in the first part applied to the lattice of subrepresentations.

Evidence that the case of quiver representation is just one example of a more general theory of asymptotics of certain gradient flows and iterated logarithms can be found in our companion paper~\cite{hkkp_itlog}.
The natural context for the considerations here should be some form of ``categorical K\"ahler geometry'', a geometric enhancement of Bridgeland's notion of stability~\cite{bridgeland07}, which the authors are developing in an ongoing project~\cite{hkkp_kaehler}.

\subsection{Background}

An $n\times n$-matrix, $A$, with complex entries is diagonalizable if and only if there is a Hermitian metric (inner product), $h$, on $\CC^n$ such that $A$ is normal, i.e. $AA^*=A^*A$, when the adjoint is taken with respect to $h$.
A generalization of this fact to quiver representations was discovered by A.D.~King~\cite{king94}.
A quiver, $Q$, is just a finite graph with oriented edges (arrows), and a representation assigns a finite-dimensional vector space $E_i$ over $\CC$ to every vertex $i$ and a linear map $\phi_\alpha:E_i\to E_j$ to every arrow $\alpha:i\to j$.
Representations of a quiver form an abelian category, so in particular the notions of simple and semisimple representation are defined in the usual way.
King's theorem then states that $E$ is semisimple if and only if one can find a metric on each $E_i$ such that
\begin{equation}\label{harmonic_metric_intro}
\sum_{\alpha}[\phi_\alpha^*,\phi_\alpha]=0.
\end{equation}
This result can be seen as a finite-dimensional analog of the celebrated Donaldson and Uhlenbeck--Yau theorem~\cite{donaldson85,uhlenbeck_yau} which relates existence of Hermitian Yang--Mills metrics on a holomorphic vector bundle to slope stability of that bundle.
Both are a special instance of the general \textit{Kempf--Ness principle}~\cite{kempf_ness} which relates quotient constructions in geometric invariant theory and symplectic geometry.

The solutions to \eqref{harmonic_metric_intro} are the minima of the function
\begin{equation}\label{riem_metric}
S(h):=\sum_{\alpha:i\to j}\mathrm{tr}(h_i^{-1}\phi_\alpha^*h_j\phi_\alpha)
\end{equation}
on the space of metrics on the given representation, where $h_i^{-1}\phi_\alpha^*h_j$ is the adjoint of $\phi_\alpha$ with respect to $h$ and $\phi_\alpha^*$ is the adjoint with respect to some arbitrary reference metric. 
Provided $E$ is semisimple we can thus follow a trajectory of the gradient flow 
\begin{equation}\label{ode_intro}
m_ih_i^{-1}\frac{dh_i}{dt}=\sum_{\alpha:i\to j}h_i^{-1}\phi_\alpha^*h_j\phi_\alpha-\sum_{\alpha:j\to i}\phi_\alpha h_j^{-1}\phi_\alpha^*h_i
\end{equation}
to arrive at a solution to \eqref{harmonic_metric_intro}. 
Here the $m_i>0$ are parameters of the homogeneous Riemmannian metric 
\[
\langle v,w\rangle:=\sum_{i\in Q_0}m_i\mathrm{tr}\left(h_i^{-1}vh_i^{-1}w\right)
\]
on the space of Hermitian metrics, $h$, on $E$.
Thus, a natural question arises: \textit{Suppose $E$ is not semisimple, then what is the asymptotic behavior of the gradient flow of the function $S$?}
Roughly what happens is that the metric grows/decays at different rates on different vectors in $E$ and this determines a filtration by subrepresentations.
It turns out that just as the existence of a minimizing metric is controlled by an algebraic criterion, semisimplicity, the asymptotic filtration, which we call \textit{iterated weight filtration}, has a purely algebraic definition which depends only on the partially ordered set of subrepresentations of $E$ together with a finite number of real parameters, and can be generalized to any finite length modular lattice.

\subsection{Weight filtration}

Before giving the general definition of the weight filtration we describe it in the case of a representation of the quiver with a single loop, i.e. just a vector space $V$ with an endomorphism $\phi$.
Without loss of generality we may assume that $\phi$ is nilpotent, otherwise consider each algebraic eigenspace separately.
The weight filtration is then the unique filtration $V_{\leq\lambda}$ of $V$ by $\lambda\in\frac{1}{2}\mathbb Z$ such that $\phi(V_{\leq\lambda})\subset V_{\leq\lambda-1}$ and $\phi^k$ induces an isomorphism $\mathrm{Gr}_{k/2}V\to\mathrm{Gr}_{-k/2}V$ for any positive integer $k$.
This is, up to relabeling, what Griffiths calls the Picard--Lefschetz filtration induced by $\phi$ in \cite{griffiths70}, see also \cite{schmid73}.
When $\phi$ is the logarithm of the unipotent part of the monodromy, it gives the weight filtration on the limiting mixed Hodge structure on the vanishing cohomology of an isolated hypersurface singularity. 
This is the origin of our terminology.

For the reader who is more familiar with abelian categories than modular lattices we state our first main result, the definition of the weight filtration (Theorem~\ref{balanced_chain_thm} in the main text), in this context.

\begin{thmdef2}\label{thm1_intro}
Suppose $\mathcal A$ is an artinian (finite length) abelian category and $X:K_0(\mathcal A)\to \RR$ a homomorphism which is positive on each class of a non-zero object.
For each object $E\in \mathcal A$ there exists a unique filtration
\begin{equation}
0=E_0\subset E_1\subset\ldots\subset E_n=E
\end{equation}
with subquotients $E_k/E_{k-1}\neq 0$ labeled by real numbers $\lambda_1<\ldots<\lambda_n$ such that the following conditions are satisfied.
\begin{enumerate}
\item
The subquotient $E_l/E_{k-1}$ is semisimple for any $1\leq k\leq l\leq n$ with $\lambda_l-\lambda_k<1$.
\item 
The balancing condition
\begin{equation}
\sum_{k=1}^n\lambda_k X(E_k/E_{k-1})=0
\end{equation}
holds.
\item
For any collection of objects $F_k$ with $E_{k-1}\subseteq F_k\subseteq E_k$, $k=1,\ldots,n$, such that $F_l/F_k$ is semisimple for $1\leq k< l\leq n$ with $\lambda_l-\lambda_k\leq 1$, the inequality
\begin{equation}
\sum_{k=1}^n\lambda_kX(F_k/E_{k-1})\leq 0
\end{equation}
holds.
\end{enumerate}
The uniquely defined filtration, depending on $X$, is called the \textit{weight filtration} on $E$.
\end{thmdef2}

There is always a canonical choice for $X$, which is to assign to each object its length (of a Jordan--H\"older filtration).
The filtration is trivial precisely when $E$ is semisimple.

We emphasize that in the case of quiver representations the filtration defined above gives only the first approximation to the \textit{iterated} weight filtration which determines the asymptotics of the gradient flow.
As the name suggests, this is a refinement of the weight filtration which is constructed by iteratively applying the above theorem/definition some finite number of times.
In order to define it, it will be much more convenient to use the language of lattices.

\subsection{Modular lattices and iterated weight filtration}

Suppose $E$ is an object in an artinian abelian category, e.g. a quiver representation.
The set $L$ of subobjects of $E$, partially order by inclusion, enjoys the following properties crucial for our purposes:
\begin{itemize}
\item
$L$ is a \textbf{lattice}: Any two elements $a,b$ have a least upper bound $a\vee b$ and greatest lower bound $a\wedge b$.
\item
\textbf{modularity}: If $a\leq b$ then $(x\wedge b)\vee a=(x\vee a)\wedge b$  for all $x\in L$.
\item
\textbf{finite length}: There is a global upper bound on the length of any chain $a_1<a_2<\ldots<a_n$ of elements in $L$.
\end{itemize}
In the first part of the paper we will work in the general context of finite length modular lattices.
Besides providing a natural level of generality, there are interesting examples of modular lattices which do not come from abelian categories, for instance normal subgroups of a finite group, or semistable subbundles (of the same slope) of a semistable Arakelov bundle.

There is an analog of the Grothendieck group $K_0$ for modular lattices.
Let $[a,b]:=\{x\in L \mid a\leq x\leq b\}$ be the interval from $a$ to $b$.
If $L$ is a modular lattice then denote by $K(L)$ the abelian group with generators $\ol{[a,b]}$, $a\leq b$, and relations
\begin{equation*}
\ol{[a,b]}+\ol{[b,c]}=\ol{[a,c]},\qquad
\ol{[a,a\vee b]}=\ol{[a\wedge b,b]}
\end{equation*}
We let $K^+(L)\subset K(L)$ be the sub-semigroup generated by elements $\ol{[a,b]}$, $a<b$.
If $L$ is moreover finite length then it is a consequence of the Jordan--H\"older--Dedekind theorem that $K(L)$ (resp. $K^+(L)$) is a free abelian group (resp. semigroup).

Note that an object in an abelian category is semisimple if and only if the corresponding lattice of subobjects is \textit{complemented}: For any $a\in L$ there is a $b\in L$ with $a\wedge b=0$ and $a\vee b=1$, where $0$ (resp. $1$) is the minimum (resp. maximum) of $L$.

With the above definitions we are ready to state our first main result in full generality.
This is Theorem~\ref{balanced_chain_thm} in the main text with some of the definitions unwrapped.

\begin{thmdef2}
Let $L\neq\emptyset$ be a finite length modular lattice and $X:K^+(L)\to \RR_{>0}$ an additive map.
Then there exists a unique filtration (=chain)
\begin{equation*}
0=a_0<a_1<\ldots<a_n=1
\end{equation*}
with intervals $[a_{k-1},a_k]\neq 0$ labeled by real numbers $\lambda_1<\ldots<\lambda_n$ such that the following conditions are satisfied.
\begin{enumerate}
\item
The interval $[a_{k-1},a_l]$ is complemented for $1\leq k\leq l\leq n$ with $\lambda_l-\lambda_k<1$.
\item 
The balancing condition
\begin{equation*}
\sum_{k=1}^n\lambda_k X([a_{k-1},a_k])=0
\end{equation*}
holds.
\item
For any collection of elements $b_k\in [a_{k-1},a_k]$, $k=1,\ldots,n$, such that $[b_k,b_l]$ is complemented for $1\leq k< l\leq n$ with $\lambda_l-\lambda_k\leq 1$, the inequality
\begin{equation*}
\sum_{k=1}^n\lambda_kX([a_{k-1},b_k])\leq 0
\end{equation*}
holds.
\end{enumerate}
The uniquely defined filtration, depending on $X$, is called the \textit{weight filtration} on $L$.
\end{thmdef2}

Given $L$ with its weight filtration $a_i,\lambda_i$, $i=1,\ldots,n$ we may construct a new finite length modular lattice $L'$ which appeared implicitly in the theorem/definition above.
Namely an element $b\in L'$ is given by elements $b_k\in [a_{k-1},a_k]$, $k=1,\ldots,n$, such that $[b_k,b_l]$ is complemented for $1\leq k< l\leq n$ with $\lambda_l-\lambda_k\leq 1$ and $\sum_{k=1}^n\lambda_kX([b_{k-1},b_k])=0$.
Moreover there is a map $X':K^+(L')\to\RR_{>0}$ given by
\[
X'([b,c]):=\sum_{k=1}^nX([b_k,c_k]).
\]
Applying the theorem/definition above to $(L',X')$ we get a filtration of $L'$ which in particular gives a filtration on each interval $[a_{k-1},a_k]$ by definition of $L'$, thus a refinement of the weight filtration on $L$ indexed by $\RR^2$ with the lexicographical order.
This inductive process continues building a sequence $(L,X),(L',X'),(L'',X''),\ldots$ until we reach a lattice which is complemented, thus has trivial weight filtration.
We call the refinement of the weight filtration constructed in this way the \textit{iterated weight filtration}, and it is indexed by $\RR^n\subset\RR^\infty$ for some $n$ with the lexicographical order.

At this point, the reader may want to skip ahead to Section~\ref{sec_example} for a discussion of the simplest example where this refinement occurs.
In examples, the refinement appears not for generic choice of $X$, but along ``walls'' in $\mathrm{Hom}(K^+(L),\RR_{>0})$ described by real algebraic varieties defined over $\mathbb Z$.
The general properties of these walls warrant further study.

We note that for any stability condition on a triangulated category in the sense of Bridgeland~\cite{bridgeland07} the full subcategory of semistable objects of a fixed phase is finite-length abelian and the restriction of the central charge determines a suitable map $X$.
Thus, our theory gives a canonical refinement of the Harder--Narasimhan filtration of any object.
One application of this refinement is perhaps to define stratifications of the stack of semistable objects by type of the weight-filtration.
Refinements of this sort were defined and studied by Kirwan~\cite{kirwan05}, in particular for vector bundles on a curve.
We do not consider stratifications in the present paper, but hope to return to this problem in the future.

\subsection{Asymptotics of the gradient flow}

For the dynamical interpretation of the iterated weight filtration we identify its indexing set, $\RR^\infty$, with the space of functions
\begin{equation*}
\RR\log t\oplus\RR\log\log t\oplus\RR\log\log\log t\oplus\ldots
\end{equation*}
defined for $t\gg 0$, as totally ordered sets.

As above, let $Q$ be a quiver, $E$ a representation of $Q$ given by vector spaces $E_i$ for each vertex $i$ and linear maps $\phi_\alpha$ for each arrow $\alpha$, and $h_i$ a Hermitian metric on $E_i$, which will be allowed to vary.
Fix a positive real number $m_i$ for each vertex $i$ of $Q$.
These determine an additive map $X:K_0(\mathrm{Rep}(Q))\to\RR$ via $X(E):=\sum_im_i\dim E_i$, as well as a Riemannian metric on the space of metrics on any representation of $Q$, see \eqref{riem_metric}.
Our general theory defines a unique iterated weight filtration $F_{\lambda}$ of $E$ labeled by $\lambda\in\RR^\infty$.
Our second main result is the following, which is Theorem~\ref{thm_asymptotics} in the main text.

\begin{theorem2}
Let $E=(E_i,\phi_\alpha)$ be a representation of a quiver $Q$ over $\CC$, then the limiting filtration of the flow \eqref{ode_intro} on metrizations of $E$ coincides with the iterated weight filtration on $E$ as an object in the category of representations of $Q$ over $\CC$ with $X$ determined by the $m_i$. 
Moreover on the piece $F_\lambda$ of the filtration, $(\lambda_1,\ldots,\lambda_n)\in\RR^n\subset\RR^{\infty}$, any trajectory $h$ of the flow satisfies
\begin{equation}
\log|h(t)|=\lambda_1\log t+\lambda_2\log\log t+\cdots+\lambda_n\log^{(n)}t+O(1)
\end{equation} 
where $\log^{(k)}$ is the $k$-times iterated logarithm.
\end{theorem2}

The proof involves an inductive procedure which produces explicit solutions of \eqref{ode_intro} up to terms in $L^1$.
A crucial property of the flow is monotonicity: If $g,h$ are solutions with $g(0)\leq h(0)$, then $g(t)\leq h(t)$ for all $t\geq 0$.

\subsection{Outline}

The text is organized as follows. 
In Section~\ref{sec_example} we discuss in detail an example which exhibits many of the general features.
This should give the reader a good idea of the practical content of our theory before diving deeper into it.
In section~\ref{sec_weight_dag} we look at the special case when all $E_i$, $i\in Q_0$, are one-dimensional, where the weight filtration can be defined much more easily as a solution to a convex optimization problem.
Section~\ref{sec_weight_lattice} concerns the purely lattice theoretic part or the work. 
After reviewing some basics, the main goal is proving existence and uniqueness of the weight filtration in any finite-length modular lattice.  
In Section~\ref{sec_flow} we construct asymptotic solutions to \eqref{ode_intro} and prove our second main result. For this, the language of $*$-algebras and $*$-bimodules provides a useful tool.

\textbf{Acknowledgments:} We thank S.~Donaldson and C.~Simpson for useful discussions.
We also thank anonymous referees for carefully proof-reading the text and providing suggestions to help improve the exposition.
The authors were supported by a Simons research grant, NSF DMS 150908, ERC Gemis, DMS-1265230, DMS-1201475 OISE-1242272 PASI, Simons collaborative Grant - HMS, HSE Megagrant, Laboratory of Mirror Symmetry NRU HSE, RF government grant, ag. 14.641.31.000, Simons Principle  Investigator Grant, CKGA VIHREN grant \foreignlanguage{russian}{КП-06-ПВ/16}.
Much of the research was conducted while the authors enjoyed the hospitality of the IHES, the IMSA Miami, and the Laboratory of Mirror Symmetry HSE Moscow


\section{Zig-zag example}
\label{sec_example}

In this section we discuss in detail the simplest example which exhibits many of the general features: refinement of the weight filtration, wall-crossing, and iterated logarithms.
This is the four-dimensional representation of the quiver
\begin{equation}
\label{zigzag_quiver}
\begin{tikzcd}
\bullet \arrow{dr} & & \bullet \arrow{dl}\arrow{dr} & \\
& \bullet & & \bullet
\end{tikzcd}
\end{equation}
which assigns $\CC$ to each vertex and the identity map to each arrow.
We hope this section will be aid the reader in following the more general discussion in subsequent sections.
In particular it would be useful to read the second subsection below before attempting Subsection~\ref{sec_ansatz}.

\subsection{Weight filtration}

The lattice $L$ of subrepresentations of the representation of the zig--zag quiver above is the set of order ideals in the partially order set $P:=\{1>2<3>4\}$, i.e. $P$ has four elements and Hasse diagram which looks like \eqref{zigzag_quiver} and $L$ is the set of subsets $I\subset P$ with the property that if $a\in I$, $b\in P$, $b\leq a$, then $b\in I$, and has Hasse diagram
\[
\begin{tikzcd}
                 &                                & P \arrow[-]{dl}\arrow[-]{dr}  & \\
                 & 124 \arrow[-]{dl}\arrow[-]{dr} &                               & 234 \arrow[-]{dl} \\
12 \arrow[-]{dr} &                                & 24 \arrow[-]{dl}\arrow[-]{dr} & \\
                 & 2 \arrow[-]{dr}                &                               & 4 \arrow[-]{dl} \\
                 &                                & \emptyset                     & 
\end{tikzcd}
\]
where we use the notation $12:=\{1,2\}$ and so on.
We have $K^+(L)=\ZZ_{>0}^P$ where e.g. $\ol{[4,124]}=(1,1,0,0)$, so $X\in\RR_{>0}^4$ under this identification.

\begin{itemize}
\item
\underline{$X_1X_4<X_2X_3$}: We claim that the weight filtration is $\emptyset<24<P$ with labels $\lambda_1=\lambda$ and $\lambda_2=\lambda+1$ where
\[
\lambda=-\frac{X_1+X_3}{X_1+X_2+X_3+X_4}\in (-1,0)
\]
as follows from the balancing condition. 
This is verified by going through the 16 possibilities for $b_1\in[\emptyset,24],b_2\in[24,P]$.
The condition $X_1X_4\leq X_2X_3$ is needed only in the case $b_1=2$, $b_2=124$.
The strict inequality $X_1X_4<X_2X_3$ ensures that $L'\cong\{0,1\}$, so there is no refinement and the iterated weight filtration is equal to the weight filtration.

\item
\underline{$X_1X_4=X_2X_3$}: The weight filtration is $\emptyset<24<P$ as in the previous case, however
\[
L'=\{(\emptyset,24)<(2,124)<(24,P)\}
\]
is not complemented and has weight filtration $(\emptyset,24)<(2,124)<(24,P)$ with labels $\mu$, $\mu+1$ where
\[
\mu=-\frac{X_3+X_4}{X_1+X_2+X_3+X_4}.
\]
There is no further refinement and the iterated weight filtration is thus $\emptyset<2<24<124<P$ with labels $(\lambda,\mu)<(\lambda,\mu+1)<(\lambda+1,\mu)<(\lambda+1,\mu+1)$.

\item
\underline{$X_1X_4>X_2X_3$}: The weight filtration is $\emptyset<2<24<124<P$ with labels
\[
\lambda=-\frac{X_1}{X_1+X_2}<\mu=-\frac{X_3}{X_3+X_4}<\lambda+1<\mu+1
\]
where $\lambda<\mu$ is equivalent to $X_1X_4>X_2X_3$. 
The lattice $L'$ has four elements and is isomorphic to the lattice of subsets of a two-element set, in particular complemented, so there is again no refinement of the weight filtration.
\end{itemize}

To summarize the situation we have a ``wall'' $X_1X_4=X_2X_3$ dividing the space $\RR_{>0}^4$ of parameters into two chambers.
The filtration is, up to relabeling, the same across a given open chamber. 
The refinement occurs only along the wall.

\subsection{Asymptotics}

Let us look at the gradient flow \eqref{ode_intro} in our 4-dimensional example.
For convenience, write the ODE in terms of variable $x_i=\log h_i\in\RR$, then we get
\begin{equation}\label{zigzag_ode1}
\begin{aligned}
m_1\dot{x}_1&=e^{x_2-x_1},\qquad m_2\dot{x}_2=-e^{x_2-x_1}-e^{x_2-x_3}, \\
m_3\dot{x}_3&=e^{x_2-x_3}+e^{x_4-x_3},\qquad m_4\dot{x}_4=-e^{x_4-x_3}.
\end{aligned}
\end{equation}
We try the ansatz
\[
x_i=a_i\log t+\log b_i
\]
which we want to solve the above equations up to error terms in $L^1(\RR_{\gg 0})$.
This is indeed possible as long as $m_1m_4\neq m_3m_2$ and discussed in detail in Section~\ref{sec_weight_dag} below.
The numbers $a_i$ come from the labels of the weight filtration for $X_i=m_i$ by our general theory.
Here we consider instead the more interesting case where $m_1m_4=m_3m_2$. 
For concreteness we take $m_1=m_2=m_3=m_4=1$.

Start by refining the first ansatz as follows:
\begin{equation}\label{zigzag_ansatz2}
\begin{aligned}
x_1&=\frac{1}{2}\log t+y_1(\log t),\qquad x_2=-\frac{1}{2}\log t+y_2(\log t), \\
x_3&=\frac{1}{2}\log t+y_3(\log t),\qquad x_4=-\frac{1}{2}\log t+y_4(\log t)
\end{aligned}
\end{equation}
This gives a system of ODEs in new dependent variables $y_i$ and independent variable $s=\log t$,
\begin{equation} \label{zigzag_ode1_5}
\begin{aligned}
\dot{y}_1&=e^{y_2-y_1}-\frac{1}{2},\qquad \dot{y}_2=-e^{y_2-y_1}-e^{y_2-y_3}+\frac{1}{2}, \\
\dot{y}_3&=e^{y_2-y_3}+e^{y_4-y_3}-\frac{1}{2},\qquad \dot{y}_4=-e^{y_4-y_3}+\frac{1}{2}.
\end{aligned}
\end{equation}
Note that $y_1$ and $y_4$ are fixed for 
\begin{equation} \label{zigzag_fp}
y_2-y_1=y_4-y_3=-\log 2.
\end{equation}
Assuming \eqref{zigzag_fp}, what remains is the system
\begin{equation} \label{zigzag_ode2}
\dot{y}_2=-e^{y_2-y_3},\qquad \dot{y}_3=e^{y_2-y_3}
\end{equation}
which has the same general form as the original one, \eqref{zigzag_ode1}.
This \textit{self-reproducing} feature of this class of equations is completely parallel to the passage from $L$ to $L'$ in the construction of the iterated weight filtration on a finite length modular lattice.

Returning to \eqref{zigzag_ode2} we easily find the explicit solution
\begin{equation} \label{zigzag_ode2_sol}
y_2=-\frac{1}{2}\log (2s),\qquad y_3=\frac{1}{2}\log (2s).
\end{equation} 
Of course \eqref{zigzag_fp} assumed that $y_1$ and $y_4$ are fixed, which contradicts \eqref{zigzag_ode2_sol}, so we do not get a solution of the original system \eqref{zigzag_ode1_5}.
However, it turns out that \eqref{zigzag_fp} and \eqref{zigzag_ode2_sol} still give the correct asymptotic behavior up to bounded terms. 
This follows from the general theory developed in the subsequent sections.
The proof involves construction a solution of \eqref{zigzag_ode1_5} up to error terms in $L^1(\RR_{\gg 0})$.
This is achieved by combining \eqref{zigzag_fp}, \eqref{zigzag_ode2_sol}, and adding terms of the form $C/s$.
Explicitly, the solution of \eqref{zigzag_ode1_5} up to terms in $L^1$ is
\begin{gather*}
y_1=-\frac{1}{2}\log\frac{s}{2}+\frac{1}{s},\qquad y_2=-\frac{1}{2}\log(2s), \\
y_3=\frac{1}{2}\log(2s),\qquad y_4=\frac{1}{2}\log\frac{s}{2}-\frac{1}{s}
\end{gather*}
which may be substituted into \eqref{zigzag_ansatz2} to give a solution to the original system \eqref{zigzag_ode1} up to terms in $L^1$.

It turns out that the strategy above, with some modifications, provides solution of the gradient flow up to terms in $L^1$ for any quiver representation, see Subsection~\ref{sec_ansatz}.
Another key ingredient, monotonicity, will be discussed there.


\section{Weights on directed acyclic graphs}
\label{sec_weight_dag}

For a special class of modular lattices constructed from directed acyclic graphs by taking the set of closed subgraphs, the weight filtration has a simpler definition avoiding the language of lattice theory.
This corresponds to the case of representations of acyclic quivers which assign a one-dimensional space to each vertex.
The discussion in this section is essentially subsumed by the later ones, and the reader is free to skip it, but we hope this section will help motivate the general theory and make it more accessible.

\subsection{Weight grading}

A \textbf{directed acyclic graph} (DAG) is an oriented graph without multiple edges or oriented cycles. 
(For us, the terms oriented graph and quiver are synonymous.)
If $G$ is a DAG, we write $G_0$ for the set of vertices and $G_1\subset G_0\times G_0$ for the set of edges/arrows.
Also write $\alpha:i\to j$ to indicate that $\alpha$ is an edge from a $i$ to $j$ where $i,j\in G_0$.
We assume throughout that the graph is finite.

An \textbf{$\RR$-grading} on a DAG, $G$, is a choice of number $v_i\in\RR$ for every vertex $i\in G_0$ which decreases at least by one on each edge, i.e. 
\begin{equation}
v_i-v_j\geq 1
\end{equation}
if there is an edge $\alpha:i\to j$. 
$\RR$-gradings form a closed convex subset in $\RR^{G_0}$.

There is a canonical ``energy minimizing'' $\RR$-grading depending only on (arbitrary) masses $m_i>0$, $i\in G_0$.
More precisely, we define the \textbf{weight grading} on $G$ for given choice of the $m_i$ to be the $\RR$-grading $v$ which minimizes 
\begin{equation}\label{grading_energy}
\sum_{i\in G_0}m_iv_i^2.
\end{equation}
Since we are minimizing essentially the length squared on a closed convex subset, existence and uniqueness of a minimizer follow for very general reasons.
The method of Lagrange multipliers (Karush--Kuhn--Tucker conditions) gives the following equivalent definition of the weight grading.

\begin{lem} \label{weight_grading_lagr}
Let $G$ be a DAG and $m_i\in\RR_{>0}$ for $i\in G_0$ arbitrary, then an $\RR$-grading, $v$, is the weight grading if it satisfies the following condition:
There are numbers $u_\alpha\geq 0$, $\alpha\in G_1$, such that $u_\alpha=0$ for any edge $\alpha:i\to j$ with $v_i-v_j>1$ and
\begin{equation}
m_iv_i=\sum_{i\xrightarrow[\alpha]{}j}u_\alpha-\sum_{k\xrightarrow[\alpha]{}i}u_\alpha
\end{equation}
for $i\in G_0$.
\end{lem}

The Lagrange multipliers $u_\alpha$ are in general not unique unless $G$ is a tree.
As a simple consequence of the lemma we see that the weight grading $v$ satisfies the balancing condition
\begin{equation}\label{weight_balancing}
\sum_{i\in G_0}m_iv_i=0.
\end{equation}
Furthermore, suppose $E\subset G_0$ is a set of vertices with the following property: If $i\in E$ and $\alpha:i\to j$ with $v_i-v_j=1$ then $j\in E$. 
Then for such subsets
\begin{equation}
\sum_{i\in E}m_iv_i=-\sum_{\substack{\alpha:k\to i \\ k\notin E, i\in E }}u_\alpha\leq 0.
\end{equation}
It turns out these properties characterize the weight grading uniquely, providing a convenient way of checking that a certain $\RR$-grading is in fact the weight grading.

\begin{prop}
Let $G$ be a DAG with choice of $m_i>0$, then an $\RR$-grading $v$ is the weight grading if and only if 
\[
\sum_{i\in G_0}m_iv_i=0.
\]
and for every subset $E\subset G_0$ such that if $i\in E$ and $\alpha:i\to j$ with $v_i-v_j=1$ then $j\in E$, then
\[
\sum_{i\in E}m_iv_i\leq 0.
\]
\end{prop}

\begin{proof}
One implication is clear from the discussion above.
Suppose then that $v$ satisfies the two conditions stated in the theorem. 
To show that $v$ is the weight grading it suffices to verify for $\delta$ in the tangent cone at $v$ to the space of $\RR$-gradings, $C$, that
\begin{equation}
\sum_{i\in G_0}m_iv_i\delta_i\geq 0
\end{equation}
i.e. the variation of \eqref{grading_energy} in the direction $\delta$ is non-negative.
Note that $C$ consists of $\delta\in\RR^{G_0}$ such that if there is an arrow $\alpha:i\to j$ and $v_i-v_j=1$ then $\delta_i\geq \delta_j$.
It follows that $C$ is generated by vectors $1_{G_0}$ and $-1_E$ where $E$ ranges over subsets of $G_0$ such that if $i\in E$ and $\alpha:i\to j$ with $v_i-v_j=1$ then $j\in E$.
By the first assumption on $v$, the balancing condition, the variation vanishes in the direction $1_{G_0}$, and by the second assumption it is non-negative in the directions $-1_E$.
This shows that $v$ is a minimum of \eqref{grading_energy}.
\end{proof}

\begin{ex}
As a basic example, consider the following DAG with $n\geq 1$ vertices:
\[
\underset{m_1}{\bullet}\longrightarrow\underset{m_2}{\bullet}\longrightarrow\cdots\longrightarrow\underset{m_n}{\bullet}
\]
The weight grading is given by $(\lambda,\lambda-1,\ldots,\lambda-n+1)$ where the highest weight $\lambda$ is determined by \eqref{weight_balancing} to be
\begin{equation}
\lambda=\frac{m_2+2m_3+\ldots+(n-1)m_n}{m_1+m_2+\ldots+m_n}
\end{equation}
Note that if $m_i=1$ for all $i$ then the weights are integers or half-integers.
For this particular graph, the only effect of changing the parameters $m_i$ is to shift the overall grading. 
We will see below that in general more interesting changes can occur along codimension one walls.
\end{ex}

\begin{remark}
One can also consider graphs with infinitely many vertices and parameters $m_i$ decaying sufficiently fast so that $\sum m_iv_i^2<\infty$ for some $\RR$-grading $v_i\in\RR$. 
Elementary Hilbert space theory then implies existence and uniqueness of an $\RR$-grading which minimizes total energy $\sum m_iv_i^2$.
\end{remark}

\subsection{Gradient flow}

The weight grading on a DAG has a dynamical interpretation, describing the asymptotics of a certain gradient flow.
Let $G$, $m_i$ be as before and fix also constants $c_\alpha>0$, $\alpha\in G_1$.
Consider the function
\begin{equation}
S:\RR^{G_0}\to\RR,\qquad S(x)=\sum_{\alpha:i\to j}c_\alpha e^{x_j-x_i}.
\end{equation}
The negative gradient flow of $S$ with respect to the flat metric
\begin{equation}
\sum_{i\in G_0}m_i(dx_i)^2
\end{equation}
is given by
\begin{equation}\label{dag_flow_vert}
m_i\dot{x}_i=\sum_{i\xrightarrow[\alpha]{}j}c_\alpha e^{x_j-x_i}-\sum_{k\xrightarrow[\alpha]{}i}c_\alpha e^{x_i-x_k}
\end{equation}

We can also write the flow in terms of variables attached to the edges instead of the vertices.
Set 
\begin{equation}
y_\alpha=-(x_j-x_i+\log c_\alpha)
\end{equation}
for each arrow $\alpha:i\to j$, then 
\begin{equation}\label{dag_flow_edge}
\dot{y}_\alpha=\frac{1}{m_i}\left(\sum_{\xleftarrow[\alpha]{}i\xrightarrow[\beta]{}}e^{-y_\beta}-\sum_{\xleftarrow[\alpha]{}i\xleftarrow[\beta]{}}e^{-y_\beta}\right)-\frac{1}{m_j}\left(\sum_{\xrightarrow[\alpha]{}j\xrightarrow[\beta]{}}e^{-y_\beta}-\sum_{\xrightarrow[\alpha]{}j\xleftarrow[\beta]{}}e^{-y_\beta}\right).
\end{equation}
The right hand side of the system of equations can be interpreted as $\Delta e^{-y}$, where $\Delta$ is a graph Laplacian.
In terms of variables $p_\alpha=e^{-y_\alpha}$ the system of equations becomes a special case of the higher--dimensional Lotka--Volterra equations which have the general form
\begin{equation}
\dot{p}_\alpha=p_\alpha\left(\sum_\beta a_{\alpha\beta}p_\beta+b_\alpha\right).
\end{equation}
This system provides a basic model for population dynamics, see for example Hofbauer--Sigmund~\cite{hofsig}. 
The asymptotic behavior in the general case can be significantly more complicated than in our case --- one need not have convergence to a stable equilibrium.

\begin{ex}
Consider the simplest non-trivial DAG:
\[
\underset{m_1}{\bullet}\longrightarrow\underset{m_2}{\bullet}
\]
The system of ODEs \eqref{dag_flow_vert} is
\begin{equation}
m_1\dot{x}_1=ce^{x_2-x_1},\qquad m_2\dot{x}_2=-ce^{x_2-x_1}
\end{equation}
with explicit solution
\begin{equation}
x_1=\frac{m_2}{m_1+m_2}\log t+\log C_1,\qquad x_2=-\frac{m_1}{m_1+m_2}\log t+\log C_2
\end{equation}
where $C_1,C_2>0$ are chosen so that
\begin{equation}
\frac{C_2}{C_1}=\frac{m_1m_2}{c(m_1+m_2)}.
\end{equation}
Note that the coefficients of $\log t$ in the solution coincide with the weight grading.
\end{ex}

In general, the ODE \eqref{dag_flow_vert} does not have an explicit solution, however it turns out that we can always find an explicit asymptotic solution which solves the equation up to terms in $L^1$. 
Such a solution will differ from an actual solution by a bounded error term, thus have the same asymptotics up to $O(1)$.

We begin with the following ansatz for the solution $x_i(t)$.
\begin{equation}
x_i=v_i\log t+b_i
\end{equation}
with $v_i,b_i\in\RR$.
Plugging this into \eqref{dag_flow_vert} gives
\begin{equation}
\frac{m_iv_i}{t}=\sum_{i\xrightarrow[\alpha]{}j}c_\alpha t^{v_j-v_i}e^{b_j-b_i}-\sum_{k\xrightarrow[\alpha]{}i}c_\alpha t^{v_i-v_k}e^{b_i-b_k}.
\end{equation}
For this equation to be true up to terms in $L^1$, it is necessary that only $t^{\leq -1}$ appear on the right hand side, i.e. $v_j-v_i\leq -1$ whenever there is an edge $\alpha:i\to j$.
Then, comparing coefficients of $t^{-1}$ (other terms are in $L^1$) we need to solve
\begin{equation}\label{ansatz1_eq}
m_iv_i=\sum_{\substack{\alpha:i\to j \\ v_i-v_j=1}}c_\alpha e^{b_j-b_i}-\sum_{\substack{\alpha:k\to i \\ v_k-v_i=1}}c_\alpha e^{b_i-b_k}.
\end{equation}
Comparing this with Lemma~\ref{weight_grading_lagr}, we see that the $v_i$ are necessarily the weight grading on $G$.
Furthermore, if \eqref{ansatz1_eq} has a solution, $b$, then we can evidently choose Lagrange multipliers $u_\alpha$ such that $u_\alpha>0$ whenever $\alpha:i\to j$ is an edge with $v_i-v_j=1$, and $u_\alpha=0$ otherwise. 
It turns out that this is not always possible.
We will see below that in the case where we cannot solve \eqref{ansatz1_eq} it is necessary to refine the original ansatz with terms involving iterated logarithms.

Note that \eqref{ansatz1_eq} is the equation for a critical point of the function
\begin{equation}
\widetilde{S}(x)=\sum_{\substack{\alpha:i\to j \\ v_i-v_j=1}}c_\alpha e^{x_j-x_i}+\sum_{i\in G_0}m_iv_ix_i
\end{equation}
Suppose that we can find $u_\alpha>0$ such that
\begin{equation}\label{lag_mult_active}
m_iv_i=\sum_{\substack{\alpha:i\to j \\ v_i-v_j=1}}u_\alpha-\sum_{\substack{\alpha:k\to i \\ v_k-v_i=1}}u_\alpha
\end{equation}
then
\begin{equation}
\widetilde{S}(x)=\sum_{\substack{\alpha:i\to j \\ v_i-v_j=1}}\left(c_\alpha e^{x_j-x_i}-u_\alpha(x_j-x_i)\right)
\end{equation}
hence $\widetilde{S}$ is the composition of a linear map (the differential $d:\RR^{G_0}\to\RR^{G_1}$) with a proper strictly convex function, thus its critical locus is an affine subspace of $\RR^{G_0}$.
We summarize the result so far in the following theorem.

\begin{thm}\label{thm_ansatz1}
Let $G$ be a DAG, $m_i>0$, and $v\in \RR^{G_0}$ the weight grading on $G$.
Suppose that we can find $u_\alpha>0$ for each edge $\alpha:i\to j$ with $v_i-v_j=1$ such that \eqref{lag_mult_active} holds.
Then the flow \eqref{dag_flow_vert} admits an asymptotic solution up to terms in $L^1$ of the form
\[
x_i=v_i\log t+b_i
\]
and thus for an actual solution $x(t)$ we have
\[
x_i=v_i\log t+O(1).
\]
\end{thm}

The claim about asymptotics of actual solutions could be verified here directly without difficulty, however we will show it for a more general setup in Subsection~\ref{sec_monotonicity}.
We conclude this subsection with an example where Theorem~\ref{thm_ansatz1} is not applicable. 
Consider the following DAG which is an orientation (zig-zag) of the $A_4$ Dynkin diagram. 
Masses $m_i$ indicate the labeling of the vertices.
\begin{equation} 
\begin{tikzpicture}[scale=.8,baseline=(current  bounding  box.center)]
\node (1) at (0,0) {$m_1$};
\node (2) at (1,-1) {$m_2$};
\node (3) at (2,0) {$m_3$};
\node (4) at (3,-1) {$m_4$};
\draw[->] (1) edge (2);
\draw[->] (3) edge (2);
\draw[->] (3) edge (4);
\end{tikzpicture}
\end{equation}
The weight grading depends on the choice of $m_1,m_2,m_3,m_4>0$.

\textbf{Case $\mathbf{m_1m_4>m_2m_3}$:}
In this region there are four distinct weights 
\begin{equation} 
\begin{tikzpicture}[scale=.8,baseline=(current  bounding  box.center)]
\node (1) at (0,0) {$\mu$};
\node (2) at (1,-1) {$\mu-1$};
\node (3) at (2,.5) {$\lambda$};
\node (4) at (3,-.5) {$\lambda-1$};
\draw[->] (1) edge (2);
\draw[->] (3) edge (2);
\draw[->] (3) edge (4);
\end{tikzpicture}
\end{equation}
where
\begin{equation}
\lambda=\frac{m_4}{m_3+m_4},\qquad \mu=\frac{m_2}{m_1+m_2}.
\end{equation}
The Lagrange multipliers which certify $v$ are
\begin{equation}
u_1=\frac{m_1m_2}{m_1+m_2},\qquad u_2=0,\qquad u_3=\frac{m_3m_4}{m_3+m_4}
\end{equation}
hence Theorem~\ref{thm_ansatz1} can be applied.

\textbf{Case $\mathbf{m_1m_4\leq m_2m_3}$:}
In this region there are two distinct weights $\lambda=v_1=v_3$ and $\lambda-1=v_2=v_4$ where
\begin{equation}
\lambda=\frac{m_2+m_4}{m_1+m_2+m_3+m_4}.
\end{equation}
The Lagrange multipliers which certify $v$ are
\begin{equation}
u_1=\frac{m_1(m_2+m_4)}{M},\qquad u_2=\frac{m_2m_3-m_1m_4}{M},\qquad u_3=\frac{(m_1+m_3)m_4}{M}
\end{equation}
where $M=m_1+m_2+m_3+m_4$.
Note that $u_2>0$ if and only if $m_1m_4<m_2m_3$, so if $m$ lies on the quadric $m_1m_4=m_2m_3$ then 
the condition of Theorem~\ref{thm_ansatz1} is not satisfied.

\subsection{From DAGs to lattices}

Given a directed acyclic graph $G$ consider the collection $L$ of subsets of $E\subset G_0$ which span \textit{closed subgraphs}, i.e. no arrows lead out of $E$.
Note that $L$ is closed under unions and intersections, thus a sublattice of the boolean lattice of all subsets of $G_0$.
We can almost recover $G$ from the partially ordered set $L$.
For example the DAGs
\begin{equation} 
\begin{tikzcd}
\bullet \arrow{r}\arrow[bend left]{rr} & \bullet \arrow{r} & \bullet & & \bullet \arrow{r} & \bullet \arrow{r} & \bullet
\end{tikzcd}
\end{equation}
have the same lattices of closed subgraphs.
However, this does not affect the weight grading.

The lattice of subrepresentations of a finite-dimensional representation is in general more complicated than the lattices constructed from graphs, in that complements, if they exist, need not be unique. 
However, such a lattice is still \textit{modular} which leads to a good theory of filtrations.
In the next section we will generalize the notion of weight filtration from DAGs to finite length modular lattices.


\section{Weight filtrations in modular lattices}
\label{sec_weight_lattice}

This section contains the proof of our first main result, the existence and uniqueness of weight filtrations in finite-length modular lattices.
The reader interested mainly in the case of quiver representations and willing to take this result on faith can skip this entire section on first reading.

In the first subsection we review some definitions and results from lattice theory.
In Subsection~\ref{sec_hn_chain} we define the Harder--Narasimhan filtration of a finite length modular lattice with polarization, as well as its mass, and prove a triangle inequality for mass.
Subsection~\ref{sec_pcrchains} introduces the concept of a \textit{paracomplemented $\RR$-filtration}, which is essential for the proof in Subsection~\ref{subsec_balanced}.
In the final subsection we discuss the iterated weight filtration and provide examples to show that it can have arbitrary depth.

\subsection{Some lattice theory basics}
\label{sec_lattices}

In this subsection we recall some basic notions from lattice theory, in particular modular lattices as introduced by Dedekind. 
We learned this material in part from G.~Birkhoff's classic textbook~\cite{birkhoff_book} and J.B.~Nation's online notes~\cite{jbnation}, which are excellent sources for more background.

A \textbf{lattice} is a partially ordered set, $L$, in which any two elements $a,b\in L$ have a least upper bound $a\vee b$ and greatest lower bound $a\wedge b$. 
When $L$ contains both a least element $0\in L$ and greatest element $1\in L$, then $L$ is called a \textbf{bound lattice}.
Given elements $a\leq b$ in $L$, the \textbf{interval} from $a$ to $b$ is the bound lattice
\begin{equation}\label{lattice_interval}
[a,b]:=\{x\in L \mid a\leq x\leq b\}.
\end{equation}
In a general lattice there are two ways of projecting an arbitrary element $x\in L$ to the interval $[a,b]$, given by the left and hand right side of the following inequality:
\begin{equation}\label{mod_inequality}
(x\wedge b)\vee a\leq (x\vee a)\wedge b.
\end{equation}
The defining property of a \textbf{modular lattice} is that the above inequality becomes an equality, hence
\begin{equation}\label{mod_law}
a\leq b\implies (x\wedge b)\vee a=(x\vee a)\wedge b\qquad\text{ for all }x\in L.
\end{equation}
The basic example of a modular lattice is the lattice of subobjects in a given object of an abelian category. 

There is an equivalence relation on the set of intervals in a modular lattice $L$ generated by
\begin{equation}
[a,a\vee b]\sim [a\wedge b,b].
\end{equation}
Modularity is equivalent to the condition that the maps
\begin{align}
[a,a\vee b]\to [a\wedge b,b],\qquad x\mapsto x\wedge b \\
[a\wedge b,b]\to [a,a\vee b],\qquad x\mapsto x\vee a
\end{align}
are inverse isomorphisms for all $a,b\in L$.
Thus, equivalent intervals are isomorphic lattices.
\begin{equation}
\begin{tikzpicture}[baseline=(current  bounding  box.center)]
\node (T) at (0,1) {$a\vee b$};
\node (L) at (-1,0) {$a$};
\node (R) at (1,0) {$b$};
\node (B) at (0,-1) {$a\wedge b$};
\draw
(T) edge (L)
(L) edge (B)
(T) edge (R)
(R) edge (B);
\node at (0,0) {\rotatebox{-45}{$\Longleftrightarrow$}};
\end{tikzpicture}
\end{equation}

A lattice $L$ is \textbf{finite length} if there is an upper bound on the length $n$ of any chain 
\begin{equation}
a_0<a_1<\ldots<a_n
\end{equation} 
of elements in $L$.
A finite length lattice is \textit{complete} in the sense that any (not necessarily finite) collection of elements has a least upper bound and greatest lower bound.
In particular, unless $L=\emptyset$, there are least and greatest elements $0$ and $1$ in any finite length lattice.
We say a lattice is \textbf{artinian} if it is modular and has finite length. 
In an artinian lattice, any two maximal chains have the same length, in fact:

\begin{thm}[Jordan--H\"older--Dedekind] \label{lattice_jh}
Suppose 
\[
0=a_0<a_1<\ldots<a_m=1,\qquad 0=b_0<b_1<\ldots<b_n=1
\]
are maximal chains in a modular lattice.
Then $m=n$ and there is a permutation $\sigma$ of the set $\{1,\ldots,n\}$ such that there are equivalences of intervals
\[
[a_{k-1},a_k]\sim [b_{\sigma(k)-1},b_{\sigma(k)}]
\]
for $k=1,\ldots,n$.
\end{thm}

The proof is essentially the same as for the classical Jordan--H\"older theorem, but translated into the setting of modular lattices.
See for example the texts mentioned at the beginning of this subsection.

Let $L$ be an artinian lattice, then we denote by $K(L)$ the abelian group with generators $\ol{[a,b]}$, $a\leq b$, and relations
\begin{align}
\ol{[a,b]}+\ol{[b,c]}=\ol{[a,c]} \\
\ol{[a,a\vee b]}=\ol{[a\wedge b,b]}
\end{align}
We let $K^+(L)\subset K(L)$ be the sub-semigroup generated by elements $\ol{[a,b]}$, $a<b$.
It is a direct consequence of Theorem~\ref{lattice_jh} that $K(L)$ (resp. $K^+(L)$) is the free abelian group (resp. semigroup) generated by the set of equivalence classes of intervals of length $1$ in $L$.

\subsection{Harder--Narasimhan filtration and mass}
\label{sec_hn_chain}

Harder--Narasimhan filtrations were originally defined for vector bundles on an algebraic curve.
The notion admits a straightforward generalization to modular lattices, which we include here for the sake of completeness and to fix terminology. 
We also prove a triangle inequality for the notion of \textit{mass} coming from the HN filtration.
Cornut~\cite{cornut17} has also recently studied Harder--Narasimhan filtrations in modular lattices by attaching building-like spaces to them.

Consider a sub-semigroup of $(\CC,+)$ of the form
\begin{equation}
H=\{re^{i\phi}\mid r>0,\phi \in I\}
\end{equation}
where $I\subset\RR$ is a half-open interval of length $\pi$, e.g. $I=[0,\pi)$.
A \textbf{polarization} on an artinian lattice $L$ is a homomorphism $Z:K(L)\to \CC$ such that $Z(K^+(L))\subset H$. 
The pair $(L,Z)$ is a \textit{polarized lattice}.
For each $a<b\in L$ we get a well-defined \textit{phase}
\begin{equation}
\phi([a,b]):=\mathrm{Arg}(Z(\ol{[a,b]}))\in I.
\end{equation}
A polarized lattice is \textbf{stable} (resp. \textbf{semistable}) if 
\begin{equation}
x\neq 0,1\implies \phi([0,x])<\phi(L)\qquad (\text{resp. }\phi([0,x])\leq\phi(L)).
\end{equation}

Note that since $Z(L)=Z([0,x])+Z([x,1])$ one has $\phi([0,x])<\phi(L)$ iff $\phi([x,1])>\phi(L)$, which gives an equivalent condition for stability.

\begin{thm}
Let $L$ be a polarized lattice, then there is a unique chain
\[
0=a_0<a_1<\ldots<a_n=1
\]
such that $[a_{k-1},a_k]$ is semistable for $k=1,\ldots,n$ and
\[
\phi([a_{k-1},a_k])>\phi([a_k,a_{k+1}]).
\]
\end{thm}

The uniquely defined chain in the theorem above is the \textbf{Harder--Narasimhan filtration}. (The terms \textit{chain} and \textit{filtration} are used interchangeably here.)

\begin{proof}
We first show uniqueness, which does not require the finite length hypothesis on $L$.
Suppose
\begin{equation}
0=a_0<a_1<\ldots<a_m=1,\qquad 0=b_0<b_1<\ldots<b_n=1
\end{equation}
are Harder--Narasimhan filtrations.
If $n=0$, then $0=1$ in $L$ and so $m=0$ also.
Otherwise, let $k$ be such that $b_1\leq a_k$ but $b_1\nleq a_{k-1}$.
This means that 
\begin{equation}
a_{k-1}<a_{k-1}\vee b_1\leq a_k,\qquad a_{k-1}\wedge b_1<b_1.
\end{equation}
By semistability of $[b_0,b_1]$ and $[a_{k-1},a_k]$ we get
\begin{equation}
\phi([b_0,b_1])\leq \phi([a_{k-1}\wedge b_1,b_1])=\phi([a_{k-1},a_{k-1}\vee b_1])\leq \phi([a_{k-1},a_k])
\end{equation}
hence, by the assumption on the slopes of the intervals, $\phi([b_0,b_1])\leq\phi([a_0,a_1])$.
By symmetry we have equality, but then $k=1$ in the above argument and thus $b_1\leq a_1$.
Again, by symmetry, it must be that $a_1=b_1$, so the proof follows by induction on $\max(m,n)$ applied to the lattice $L'=[a_1,1]=[b_1,1]$.

Next we show existence, excluding the trivial case where $0=1$. 
It follows from the finite length hypothesis that the set of complex numbers $Z([0,a])$, $a>0$, is finite, so let
\begin{equation}
\phi:=\max\{\phi([0,a])\mid a>0\}
\end{equation}
and $a_1$ be the join of all $a>0$ with $\phi([0,a])=\phi$.
By construction $[0,a_1]$ is semistable, and furthermore any interval $[a_1,a]$, $a_1<a$, must satisfy $\phi([a_1,a])<\phi$ by maximality.
Thus, if the process is continued inductively with $[a_1,1]$, then the $\phi([a_k,a_{k+1}])$ are strictly decreasing.
\end{proof}

If $(L,Z)$ is a polarized lattice with HN-filtration $a_0<a_1<\ldots<a_n$ then the \textbf{mass} of $L$ is defined as
\begin{equation}\label{lattice_mass}
m(L)=\sum_{k=1}^n|Z([a_{k-1},a_k])|.
\end{equation}
It follows from the triangle inequality that
\begin{equation}\label{bps_ineq}
m(L)\geq |Z(L)|
\end{equation}
with equality if and only if $L$ is semistable.
The mass satisfies the following triangle inequality.

\begin{thm}\label{mass_triangle_ineq}
If $(L,Z)$ is a polarized lattice then
\[
m(L)\leq m([0,x])+m([x,1])
\]
for any $x\in L$. 
More generally, by induction, if $0=a_0<a_1<\ldots<a_n=1$ is any chain in $L$  then
\begin{equation}
m(L)\leq \sum_{k=1}^n m([a_{k-1},a_k]).
\end{equation}
\end{thm}

\begin{proof}
First consider the case when $[0,x]$ and $[x,1]$ are semistable.
Let $\phi_1=\phi([0,x])$ and $\phi_2=\phi([x,1])$.
If $\phi_1>\phi_2$, then $0<x<1$ is a HN filtration, and there is nothing to show.
If $\phi_1=\phi_2$, then $L$ is semistable and $m(L)=m([0,x])+m([x,1])$.
If $\phi_1<\phi_2$ let $0=a_0<a_1<\ldots<a_n=1$ be the HN filtration of $L$.
In this case
\begin{equation}
\phi_2\geq\phi([a_0,a_1])>\ldots>\phi([a_{n-1},a_n])\geq \phi_1.
\end{equation}

Indeed to see that $\phi_2\geq\phi([a_0,a_1])$ suppose $x<x\vee a_1$, then by semistablity
\begin{equation}
\phi_2=\phi([x,1])\geq\phi([x,x\vee a_1])=\phi([x\wedge a_1,a_1])\geq \phi([0,a_1])
\end{equation} 
and otherwise $a_1\leq x$ so $\phi([0,a_1])\leq \phi([0,x])=\phi_1<\phi_2$.

To show the inequality, let $z=Z([0,x])$ and $w=Z([x,1])$ which form an $\RR$-basis of $\CC$ by assumption.
If we write
\begin{equation}
Z([a_{k-1},a_k])=\lambda_k z+\mu_k w
\end{equation}
then $\lambda_k,\mu_k\geq 0$ by the bound on the phases.
Thus
\begin{equation}
|Z([a_{k-1},a_k])|\leq \lambda_k|z|+\mu_k|w|
\end{equation}
and taking the sum over all $k$ we get
\begin{equation}
m(L)\leq |z|+|w|=m([0,x])+m([x,1])
\end{equation}
since $\sum_k\lambda_k=\sum_k\mu_k=1$.

The general case is equivalent to the claim that if $0=a_0<a_1<\ldots<a_n=1$ is any chain in $L$ with $[a_{k-1},a_k]$ semistable, then
\begin{equation}
m(L)\leq \sum_{k=1}^n|Z([a_{k-1},a_k])|=:M
\end{equation}
since we get such a chain by concatenating the HN filtrations of $[0,x]$ and $[x,1]$.
The strategy is to modify the chain step-by-step until it becomes the HN filtration, with $M$ getting smaller each time.

If there are two consecutive intervals in the chain with the same phase, then they can be combined to a single interval, decreasing the length of chain by one.
If after this the chain $a_k$ is not the HN filtration, then there must be consecutive intervals with
\begin{equation}
\phi([a_{k-1},a_k])<\phi([a_k,a_{k+1}]).
\end{equation}
If $a_{k-1}<a_k<a_{k+1}$ is replaced by the HN filtration of $[a_{k-1},a_{k+1}]$, then by the first part of the proof $M$ either gets strictly smaller or the length of the chain stays the same.
Either way we must eventually reach the HN filtration, since the possible values of $M$ form a discrete subset of $\RR_{\geq 0}$, and if $M$ remains constant then the phases will eventually be in the right order.
\end{proof}

\subsection{Paracomplemented $\RR$--filtrations}
\label{sec_pcrchains}

Let $L$ be an artinian lattice.
An $\RR$-filtration in $L$ is a strictly increasing sequence of elements in $L$ ending with $1$ and labeled by real numbers.
The following notation will be convenient.
Given a finite subset $X$ of $\RR$ let $I_0,\ldots,I_n$ be the connected components of the complement $\RR\setminus X$ in their natural order. 
Any chain
\begin{equation}
0=a_0<a_1<\ldots<a_n=1
\end{equation}
in $L$ defines a locally constant increasing function $a:\RR\setminus X\to L$.
Let $a_+,a_-:\RR\to L$ be the upper-/lower-semicontinuous extensions of $a$, then we call this pair of increasing functions an \textbf{$\mathbf{\RR}$-filtration} in $L$.
Thus an $\RR$-filtration in $L$ is a pair of increasing functions $a_\pm:\RR\to L$ with $a_+$ upper-semicontinuous, $a_-$ lower semicontinuous, $a_+=a_-$ outside a finite set, and $a_\pm(\lambda)=0$ for $\lambda\ll 0$ and $a_\pm(\lambda)=1$ for $\lambda\gg 0$. 
Of course any one of $a_\pm$ determines the other, but it will be convenient to have both.
The \textbf{support} of an $\RR$-filtration is the finite set
\begin{equation}
\mathrm{supp}(a)=\{\lambda\in\RR\mid a_+(\lambda)\neq a_-(\lambda)\}.
\end{equation}

A lattice $L$ with $0,1$ is \textbf{complemented} if any $a\in L$ has a \textbf{complement}: An element $b\in L$ with 
\begin{equation}
a\wedge b=0,\qquad a\vee b=1.
\end{equation}
Note that for the lattice of subobjects in a given object $E$  of an artinian category, the property being complemented means that $E$ is semisimple.

We call an $\RR$-filtration, $a$, \textbf{paracomplemented} if all intervals $[a_+(\lambda),a_+(\lambda+1)]$, $\lambda\in\RR$, are complemented lattices. 
Equivalently, all intervals $[a_-(\lambda),a_-(\lambda+1)]$ are complemented lattices.

Let $\mathcal B(L)$ be the set of all paracomplemented $\RR$-filtrations in $L$.
Denote by $C_0(\RR;K(L))$ the abelian group of finite $K(L)$-linear combinations of points in $\RR$, with the obvious topology coming from $\RR$.
We can introduce a topology in $\mathcal B(L)$ such that the map
\begin{equation}\label{prc_cl}
\mathrm{cl}:\mathcal B(L)\to C_0(\RR;K(L)),\qquad a\mapsto \sum_{\lambda\in\RR}\ol{[a_-(\lambda),a_+(\lambda)]}\lambda
\end{equation}
is continuous.
A neighborhood basis at $a\in \mathcal B(L)$, is given by sets
\begin{equation}
U_r(a)=\left\{b\in\mathcal B(L)\mid \mathrm{dist}(\lambda,\mathrm{supp}(a))\geq r\implies a_\pm(\lambda)=b_\pm(\lambda)\right\}
\end{equation}
where $r>0$.
This topology is Hausdorff, but generally not locally compact for infinite $L$.

We will describe the local structure around $a\in\mathcal B(L)$ in terms of another artinian lattice, $\Lambda(a)$.
By definition, an element $x\in\Lambda(a)$ is given by $x_\lambda\in[a_-(\lambda),a_+(\lambda)]$ such that $a_+(\lambda)\in [x_\lambda,x_{\lambda+1}]$ has a complement for every $\lambda\in\RR$.
Thus, $\Lambda(a)$ is a subset of
\begin{equation}
\prod_{\lambda\in\RR}[a_-(\lambda),a_+(\lambda)]
\end{equation}
which is an artinian lattice as an essentially finite product of such.
However, it is not obvious that $\Lambda(a)$ is closed under $\vee$ and $\wedge$, i.e. is a sublattice.
Showing this will require a lemma about complements.

We draw a diagram
\begin{equation}
\begin{tikzpicture}[baseline=(current  bounding  box.center)]
\node (T) at (0,1) {$d$};
\node (L) at (-1,0) {$a$};
\node (R) at (1,0) {$b$};
\node (B) at (0,-1) {$c$};
\draw
(T) edge (L)
(L) edge (B)
(T) edge (R)
(R) edge (B);
\end{tikzpicture}
\end{equation}
to represent the statement that $c=a\wedge b$ and $d=a\vee b$, i.e. that $a$ has complement $b$ in $[c,d]$.
These diagrams satisfy \textit{cut and paste rules}:
\begin{equation}
\begin{tikzpicture}[scale=.8,baseline=(current  bounding  box.center)]
\node (T) at (0,1) {$\mathbf d$};
\node (L) at (-1,0) {$a$};
\node (R) at (1,0) {$\mathbf b$};
\node (B) at (0,-1) {$c$};
\draw
(T) edge (L)
(L) edge (B)
(T) edge (R)
(R) edge (B);
\node at (2,0) {and};
\node (T) at (4,1) {$e$};
\node (L) at (3,0) {$\mathbf d$};
\node (R) at (5,0) {$f$};
\node (B) at (4,-1) {$\mathbf b$};
\draw
(T) edge (L)
(L) edge (B)
(T) edge (R)
(R) edge (B);
\node at (6,0) {$\implies$};
\node (T) at (8,1) {$e$};
\node (L) at (7,0) {$a$};
\node (R) at (9,0) {$f$};
\node (B) at (8,-1) {$c$};
\draw
(T) edge (L)
(L) edge (B)
(T) edge (R)
(R) edge (B);
\end{tikzpicture}
\end{equation}
\begin{equation}
\begin{tikzpicture}[scale=.8,baseline=(current  bounding  box.center)]
\node (T) at (0,1) {$d$};
\node (L) at (-1,0) {$a$};
\node (R) at (1,0) {$b$};
\node (B) at (0,-1) {$c$};
\draw
(T) edge (L)
(L) edge (B)
(T) edge (R)
(R) edge (B);
\node at (3.5,0) {and $c\leq \mathbf x\leq b\implies$};
\node (T) at (7,1) {$\mathbf{a\vee x}$};
\node (L) at (6,0) {$a$};
\node (R) at (8,0) {$\mathbf x$};
\node (B) at (7,-1) {$c$};
\draw
(T) edge (L)
(L) edge (B)
(T) edge (R)
(R) edge (B);
\node (T) at (11,1) {$d$};
\node (L) at (10,0) {$\mathbf{a\vee x}$};
\node (R) at (12,0) {$b$};
\node (B) at (11,-1) {$\mathbf x$};
\draw
(T) edge (L)
(L) edge (B)
(T) edge (R)
(R) edge (B);
\end{tikzpicture}
\end{equation}
In the following we will not draw all the diagrams for practical reasons, but they proved a useful device to avoid getting lost in formulas.

\begin{lem}\label{compl_in_union}
Let $L$ be a modular lattice with $0$. 
Suppose $x,a,b\in L$ such that $x\leq a\wedge b$, $x$ has a complement in $[0,a]$ and $[0,b]$, and $[x,a]$ is complemented, then $x$ has a complement in $[0,a\vee b]$.
\end{lem}

\begin{proof}
Let $a'$ be a complement of $a\wedge b$ in $[x,a]$, then we have a diagram
\begin{equation}
\begin{tikzpicture}[scale=.8,baseline=(current  bounding  box.center)]
\node (T) at (0,1) {$a\vee b$};
\node (L) at (-1,0) {$a'$};
\node (R) at (1,0) {$b$};
\node (B) at (0,-1) {$x$};
\draw
(T) edge (L)
(L) edge (B)
(T) edge (R)
(R) edge (B);
\end{tikzpicture}
\end{equation}
and $x$ has a complement $c$ in $[0,a']$.
If $d$ is a complement of $x$ in $[0,b]$, then we have a diagram
\begin{equation}
\begin{tikzpicture}[scale=.8,baseline=(current  bounding  box.center)]
\node (T) at (0,1) {$a\vee b$};
\node (L) at (-1,0) {$c\vee d$};
\node (R) at (1,0) {$b$};
\node (B) at (0,-1) {$d$};
\draw
(T) edge (L)
(L) edge (B)
(T) edge (R)
(R) edge (B);
\end{tikzpicture}
\end{equation}
which shows that $c\vee d$ is a complement of $x$ in $[0,a\vee b]$.
\end{proof}

\begin{prop}
If $a\in\mathcal B(L)$ then $\Lambda(a)\subset\prod[a_-(\lambda),a_+(\lambda)]$ is a sublattice, hence an artinian lattice.
\end{prop}

\begin{proof}
Suppose $x,y\in\Lambda(a)$, then $a_+(\lambda)$ has a complement in both $[x_\lambda\vee y_\lambda,x_{\lambda+1}]$ and $[x_\lambda\vee y_\lambda,y_{\lambda+1}]$, and $[a_+(\lambda),x_{\lambda+1}]$ is complemented.
By the lemma, $a_+(\lambda)$ has a complement in $[x_\lambda\vee y_\lambda,x_{\lambda+1}\vee y_{\lambda+1}]$, hence $x\vee y\in \Lambda(a)$.
By the dual argument, $\Lambda(a)$ is also closed under $\wedge$.
\end{proof}

The defining condition for elements in $\Lambda(a)$ can be reformulated.

\begin{lem}
Let $a\in\mathcal B(L)$, $x\in\prod[a_-(\lambda),a_+(\lambda)]$, $\lambda\in\RR$,
then the following are equivalent:
\begin{enumerate}
\item 
$a_+(\lambda)$ has a complement in $[x_\lambda,x_{\lambda+1}]$.
\item
$a_-(\lambda+1)$ has a complement in $[x_\lambda,x_{\lambda+1}]$.
\item
$[x_\lambda,x_{\lambda+1}]$ is complemented.
\end{enumerate}
These conditions hold for all $\lambda\in\RR$ if and only if $x\in\Lambda(a)$.
\end{lem}

This is a direct consequence of the following lemma.

\begin{lem}
Let $L$ be a bound modular lattice, $x\in L$ such that $[0,x]$ and $[x,1]$ are complemented, and $x$ has a complement in $L$.
Then $L$ is complemented.
\end{lem}

\begin{proof}
Let $a\in L$, $b$ a complement of $a\wedge x\in[0,x]$, $c$ a complement of $a\vee x\in[x,1]$, and $d$ a complement of $x$ in $L$, then we have the following diagram.
\begin{equation}
\begin{tikzpicture}[scale=.8,baseline=(current  bounding  box.center)]
\node (1) at (0,0) {$a$};
\node (2) at (1,1) {$a\vee x$};
\node (3) at (2,2) {$1$};
\node (4) at (1,-1) {$a\wedge x$};
\node (5) at (2,0) {$x$};
\node (6) at (3,1) {$c$};
\node (7) at (2,-2) {$0$};
\node (8) at (3,-1) {$b$};
\node (9) at (4,0) {$b\vee (c\wedge d)$};
\draw
(1) edge (2) (2) edge (3) (4) edge (5) (5) edge (6) (7) edge (8) (8) edge (9)
(1) edge (4) (2) edge (5) (3) edge (6) (4) edge (7) (5) edge (8) (6) edge (9);
\end{tikzpicture}
\end{equation}
This shows that $a$ has complement $b\vee (c\wedge d)$ in $L$.
\end{proof}

Note that the defining condition for $x$ to be in $\Lambda(a)$ only relates $x_\lambda$ and $x_{\lambda+1}$, so $x_\lambda$ and $x_\mu$ are completely independent if $\lambda-\mu$ is not an integer.
Hence $\Lambda(a)$ splits as a product
\begin{gather}
\Lambda(a)=\prod_{\tau\in\RR/\ZZ}\Lambda_\tau(a) \label{loc_lattice_splitting} \\
\Lambda_\tau(a)\subset\prod_{\lambda\in\tau}[a_-(\lambda),a_+(\lambda)].
\end{gather}

For $a\in\mathcal B(L)$ define
\begin{gather}
\rho_1(a)=\min\left\{|\lambda-\mu|\mid \lambda\neq\mu,\lambda,\mu\in\mathrm{supp}(a)\right\} \\
\rho_2(a)=\min\left\{|\lambda-\mu|-1\mid|\lambda-\mu|>1,\lambda,\mu\in\mathrm{supp}(a)\right\} \\
\rho(a)=\frac{1}{2}\min(\rho_1(a),\rho_2(a),1)>0.
\end{gather}
The following gives a local description of $\mathcal B(L)$.

\begin{prop}\label{prc_space_local}
Let $a\in\mathcal B(L)$, $\rho=\rho(a)$, then there is a canonical bijection between $U_\rho(a)$ and the set of $\RR$-filtrations in $\Lambda(a)$ with support in $(-\rho,\rho)$.
\end{prop}

\begin{proof}
Denote the set of $\RR$-filtrations in $\Lambda(a)$ with support in $(-\rho,\rho)$ by $V$.
The map $U_\rho(a)\to V$ sends $b\in U_\rho(a)$ to the $\RR$-chain $x\in\Lambda(a)$ with
\begin{equation}\label{prc_def}
x_\pm(\alpha)_\lambda=b_\pm(\alpha+\lambda)\quad\in [a_-(\lambda),a_+(\lambda)]
\end{equation}
for $\alpha\in [-\rho,\rho]$ and $\lambda\in\mathrm{supp}(a)$.
To see that $x_\pm(\alpha)\in\Lambda(a)$ note that if $\lambda,\lambda+1\in\mathrm{supp}(a)$ then
\begin{equation}
[x_+(\alpha)_\lambda,x_+(\alpha)_{\lambda+1}]=[b_+(\alpha+\lambda),b_+(\alpha+\lambda+1)]
\end{equation}
is complemented since $b$ is paracomplemented by assumption.
For this part we only used $\rho\leq \rho_1/2$, not $\rho\leq\rho_2/2$.

The inverse map $V\to U_\rho(a)$ sends $x\in V$ to $b\in U_\rho(a)$ with the same relation \eqref{prc_def}.
We need to check that $b$ is paracomplemented.
So suppose that
\begin{equation}
\alpha+\lambda+1=\alpha'+\lambda',\qquad \lambda,\lambda'\in\mathrm{supp}(a),\qquad\alpha,\alpha'\in(-\rho,\rho).
\end{equation}
Because of $\rho\leq 1/2$ we have $\lambda<\lambda'$ also.
We need to show that 
\begin{equation}
[b_+(\alpha+\lambda),b_+(\alpha'+\lambda')]=[x_+(\alpha)_\lambda,x_+(\alpha')_{\lambda'}]
\end{equation}
is complemented.
If $\lambda'-\lambda>1$ the by definition of $\rho$ we get $2\rho\leq \lambda'-\lambda-1$ hence
\begin{equation}
\alpha+\lambda+1<\rho+\lambda+1\leq-\rho+\lambda'<\alpha'+\lambda'
\end{equation}
which is a contradiction, thus $\lambda'-\lambda\leq 1$.
If $\lambda'-\lambda<1$, then 
\begin{equation}
[x_+(\alpha)_\lambda,x_+(\alpha')_{\lambda'}]\subset [a_-(\lambda),a_+(\lambda')]
\end{equation}
which is complemented since $a$ is paracomplemented.
Otherwise $\lambda'=\lambda+1$ so $\alpha=\alpha'$, but then the interval is complemented because $x_+(\alpha)\in\Lambda(a)$.
\end{proof}

If $a\in\mathcal B(L)$ and $x$ is an $\RR$-filtration in $\Lambda(a)$ with support in $(-\rho,\rho)$, $\rho=\rho(a)$, and $b\in\mathcal B(L)$ corresponds to $x$, then $\Lambda(b)$ splits as a product
\begin{equation}\label{nearby_lattice_split}
\Lambda(b)=\prod_{\lambda\in\RR}[x_-(\lambda),x_+(\lambda)]
\end{equation}
which follows from \eqref{loc_lattice_splitting} and the definition of $\rho$.
Essentially, as $a$ is deformed to $b$ classes of the support in $\RR/\ZZ=S^1$ split but do not collide.

\subsection{Weight filtrations}
\label{subsec_balanced}

In this section we define a weight-type filtration in any finite length modular lattice by proving an existence and uniqueness theorem.

Let $L$ be an artinian lattice and let $X:K(L)\to\RR$ be a homomorphism with $X(K^+(L))\subset \RR_{>0}$.
For any $a\in\mathcal B(L)$ the lattice $\Lambda(a)$ has a canonical polarization given by
\begin{equation}
Z([x,y])=\sum_{\lambda\in\mathrm{supp}(a)}(1+\lambda i)X([x_\lambda,y_\lambda])
\end{equation}
for $x,y\in\Lambda(a)$, $x\leq y$.
The main result of this section is the following.

\begin{thm}\label{balanced_chain_thm}
Let $L$ be an artinian lattice and $X:K^+(L)\to\RR_{>0}$ a semigroup homomorphism.
Then there exists a unique $a\in\mathcal B(L)$ such that $\Lambda(a)$ is semistable with phase $\phi(\Lambda(a))=0$.
\end{thm}

We call the paracomplemented $\RR$-filtration in $L$ which is uniquely determined by the theorem the \textbf{weight filtration} in $L$.

The theorem stated in the introduction is the special case where $L$ is the lattice of subobjects of a fixed object $E\in\mc A$ in an artinian abelian category $\mc A$.
Then $K^+(L)$ is the sub-semigroup of $K(\mc A)$ generated by simple objects which appear as sub-quotients of $E$ and we can obtain $X$ as in the theorem above by restriction.
The first condition of the theorem in the introduction is that the filtration is paracomplemented, the second that $\phi(\Lambda(a))=0$, and the third that $\Lambda(a)$ is semistable.

\begin{proof}
\ul{Uniqueness.} Suppose $a$ and $b$ are $\RR$-filtrations in $L$.
Combine these to
\begin{gather}
c_+(\alpha,\beta):=a_+(\alpha)\wedge b_+(\beta) \\
c_-(\alpha,\beta):=a_+(\alpha)\wedge b_+(\beta)\wedge (a_-(\alpha)\vee b_-(\beta)).
\end{gather}
We claim that $a=b$ if and only if $c_+(\alpha,\beta)=c_-(\alpha,\beta)$ for all $\alpha\neq\beta$.
In one direction, if $a=b$, $\alpha<\beta$ say, then
\begin{equation}
a_+(\alpha)\wedge a_+(\beta)\wedge (a_-(\alpha)\vee a_-(\beta))=a_+(\alpha)=a_+(\alpha)\wedge a_+(\beta).
\end{equation}
On the other hand, if $a\neq b$ then there is an $\alpha\in\RR$ with $a_-(\alpha)=b_-(\alpha)$ but $a_+(\alpha)\neq b_+(\alpha)$. 
By symmetry, we may assume that $a_+(\alpha)\nleq b_+(\alpha)$, so there is a $\beta>\alpha$ with
\begin{equation}
a_+(\alpha)\leq b_+(\beta),\qquad a_+(\alpha)\nleq b_-(\beta).
\end{equation}
We have
\begin{equation}
a_+(\alpha)\wedge b_+(\beta)=a_+(\alpha)>a_+(\alpha)\wedge b_-(\beta)=a_+(\alpha)\wedge b_+(\beta)\wedge (a_-(\alpha)\vee b_-(\beta)) 
\end{equation}
thus $c_+(\alpha,\beta)>c_-(\alpha,\beta)$.

Now suppose $a,b\in\mathcal B(L)$ are both semistable of phase $0$, and $a\neq b$ for contradiction.
Then there are $\alpha\neq\beta$ with $c_-(\alpha,\beta)<c_+(\alpha,\beta)$ and we may assume by symmetry that such a pair exists with $\alpha>\beta$.
Let $\delta>0$ be maximal such that there exists an $\alpha$ with with 
\begin{equation}
c_-(\alpha,\alpha-\delta)<c_+(\alpha,\alpha-\delta)
\end{equation} 
i.e. the most off--diagonal.
Such a $\delta$ exists because of finiteness of the filtrations.
We claim that
\begin{equation}
x_\alpha=a_-(\alpha)\vee b_+(\alpha-\delta)
\end{equation}
defines an element of $\Lambda(a)$.
First, by choice of $\delta$, we have $b_+(\alpha-\delta)\leq a_+(\alpha)$ thus $x_\alpha\in[a_-(\alpha),a_+(\alpha)]$.
We need to show that $a_+(\alpha)$ has a complement in $[x_\alpha,x_{\alpha+1}]$.
Consider the following diagrams
\begin{equation}
\begin{tikzpicture}[scale=1,baseline=(current  bounding  box.center)]
\node (1) at (0,.5) {$a_+(\alpha)$};
\node (2) at (1,1.5) {$a_+(\alpha)\vee b_+(\alpha-\delta+1)$};
\node (3) at (1,-.5) {$a_+(\alpha)\wedge b_+(\alpha-\delta+1)\qquad$};
\node (4) at (2,.5) {$b_+(\alpha-\delta+1)$};
\node (5) at (2,-1.5) {$b_+(\alpha-\delta)$};
\node (6) at (3,-.5) {$d$};
\draw (1) edge (2) (3) edge (4) (5) edge (6) (1) edge (3) (3) edge (5) (2) edge (4) (4) edge (6);
\node (7) at (5,0) {$a_+(\alpha)$};
\node (8) at (6,1) {$a_-(\alpha+1)$};
\node (9) at (6,-1) {$a_-(\alpha)\vee b_+(\alpha-\delta)=x_\alpha$};
\node (10) at (7,0) {$e$};
\draw (7) edge (8) (9) edge (10) (7) edge (9) (8) edge (10);
\end{tikzpicture}
\end{equation}
\begin{equation}
\begin{tikzpicture}[scale=1,baseline=(current  bounding  box.center)]
\node (1) at (0,.5) {$a_+(\alpha)$};
\node (2) at (1,1.5) {$a_+(\alpha)\vee b_+(\alpha-\delta+1)$};
\node (3) at (1,-.5) {$a_-(\alpha)\vee b_+(\alpha-\delta)=x_\alpha$};
\node (4) at (2,.5) {$a_-(\alpha)\vee d$};
\draw (1) edge (2) (3) edge (4) (1) edge (3) (2) edge (4);
\end{tikzpicture}
\end{equation}
where existence of complements $e,d$ follows from the assumption that $a,b$ are paracomplemented, and the third is obtained from the first.
Since 
\begin{equation}
a_+(\alpha)\vee a_-(\alpha+1)\vee b_+(\alpha-\delta+1)=x_{\alpha+1}
\end{equation}
the claim follows from Lemma~\ref{compl_in_union}.

In a similar way one shows that
\begin{equation}
y_\beta=a_-(\beta+\delta)\wedge b_+(\beta)
\end{equation}
defines an element $y\in\Lambda(b)$.
We compute
\begin{align}
\mathrm{Im}Z([0,x])&=\sum_\alpha \alpha X([a_-(\alpha),x_\alpha]) \\
&>\sum_\beta \beta X([a_-(\beta+\delta),x_{\beta+\delta}]) \\
&=\sum_\beta \beta X([a_-(\beta+\delta)\wedge b_+(\beta),b_+(\beta)]) \\
&=\sum_\beta \beta (X([b_-(\beta),b_+(\beta)])-X([b_-(\beta),y_\beta])) \\
&=0-\mathrm{Im}Z([0,y])
\end{align}
which implies that at least one of $\mathrm{Im}Z([0,x])$, $\mathrm{Im}Z([0,y])$ is positive.
This contradicts the assumption that both $\Lambda(a)$ and $\Lambda(b)$ are semistable.

\ul{Existence.}
Consider the function
\begin{equation}
m:\mathcal B(L)\to\RR,\qquad a\mapsto m(\Lambda(a))
\end{equation}
sending a paracomplemented $\RR$-filtration, $a$, to the mass of the associated lattice $\Lambda(a)$.
By \eqref{bps_ineq} we have
\begin{equation}
m(\Lambda(a))\geq|Z(\Lambda(a))|\geq \mathrm{Re}Z(\Lambda(a))=X(L)
\end{equation}
with equality if and only if $\Lambda$ is semistable of phase $0$, i.e. the weight filtration.

We claim that if $a\in\mathcal B(L)$ is a local minimum of $m$, then $a$ is a weight filtration, thus a global minimum.
Suppose $x_0<x_1<\ldots<x_n$ is the HN filtration in $\Lambda(a)$, $\phi_k:=\phi([x_{k-1},x_k])$.
We want to show that $n=1$ and $\phi_1=0$ if $a$ is a local minimum.
The idea is to deform $a$ using its HN filtration.
Let $t>0$ and consider the $\RR$-filtration in $\Lambda(a)$ with support
\begin{equation}
-\phi_1t<\ldots<-\phi_nt
\end{equation}
and values $x_0<x_1<\ldots<x_n$.
For sufficiently small $t>0$ this $\RR$-filtration has support in $(-\rho(a),\rho(a))$, so let $a_t\in\mathcal B(L)$ be the corresponding paracomplemented $\RR$-filtration given by Proposition~\ref{prc_space_local}.
We have $a_t\to a$ as $t\to 0$.
The mass of $a_t$ is given by
\begin{gather}
m(\Lambda(a_t))=\sum_{k=1}^n|z_k(t)| \\
z_k(t):=\sum_{\lambda\in\RR}(1+(\lambda-\phi_kt)i)X([x_{k-1,\lambda},x_{k,\lambda}])
\end{gather}
which also gives the mass of $a$ for $t=0$.
Note that $\mathrm{Re}(z_k(t))$ is independent of $t$ and
\begin{equation}
\frac{d\mathrm{Im}(z_k)}{dt}=-\phi_k\sum_{\lambda\in\RR}X([x_{k-1,\lambda},x_{k,\lambda}])
\end{equation}
which has the opposite sign of $\phi_k$, if $\phi_k\neq 0$.
But $\phi_k=\mathrm{Arg}(Z([x_{k-1},x_k]))$ has the same sign as 
\begin{equation}
\mathrm{Im}(z_k(0))=\mathrm{Im}(Z([x_{k-1},x_k]))
\end{equation}
hence $a$ cannot be a local minimum unless $\phi_k=0$ for all $k$, i.e. $n=1$ and $\Lambda(a)$ is semistable of phase $0$.

In preparation for what follows, we want to show that there is a $C>0$ such that
\begin{equation}\label{supp_bound}
\max\{|\lambda|,\lambda\in\mathrm{supp}(a)\}\leq Cm(\Lambda(a))
\end{equation}
for any $a\in\mathcal B(L)$.
The argument is that the cardinality of $\mathrm{supp}(a)$ is bounded above by the length, $n$, of $L$, so if the diameter of $\mathrm{supp}(a)$ becomes larger than $n-1$, then there is a gap of length $>1$ and $\Lambda(a)$ splits as a product corresponding to points on the left and right of the gap.
Thus, if the left hand side of \eqref{supp_bound} is larger than $n-1$, then there must be a factor of $\Lambda(a)$ (possibly everything) supported entirely on one side of $0\in\RR$.
The mass of this factor is bounded above by $X_{\mathrm{min}}$ times the distance of its support to $0$, where $X_{\mathrm{min}}$ is the minimum of $X$ on $K^+(L)$.
Note also that since $m$ is bounded below by a positive constant, any additive constant in the estimate can be absorbed into $C$.

Now recall from \eqref{prc_cl} that there is continuous map
\begin{equation}
\mathrm{cl}:\mathcal B(L)\to C_0(\RR;K(L)),\qquad a\mapsto \sum_{\lambda\in\RR}\ol{[a_-(\lambda),a_+(\lambda)]}\lambda
\end{equation}
whose image is contained in the homology class in $H_0(\RR;K(L))=K(L)$ given by $[L]$. 
By \eqref{supp_bound} the infimum of $m:\mathcal B(L)\to\RR$ stays the same if we restrict to a subset $V\subset\mathcal B(L)$ given by $\RR$-filtrations supported in $[-M,M]$ for some sufficiently large $M\gg 0$.
The image of $V$ under $\mathrm{cl}$ is contained in the set $W$ of $0$-chains supported in $[-M,M]$, with coefficients in $K^+(L)$, and with class $[L]\in H_0(\RR;K(L))$, which is compact.
In fact, $\mathrm{cl}(V)\subset W$ is closed, hence compact.
To see this, suppose $x_n\in V$ with $\mathrm{cl}(x_n)\to y\in W$.
If $Y=\mathrm{supp}(y)$ let 
\begin{equation}\label{rad2}
\kappa=\frac{1}{2}\min\{|\lambda-\mu|\mid \lambda,\mu\in Y\}\cup\{1-|\lambda-\mu|\mid \lambda,\mu\in Y, |\lambda-\mu|<1\}>0
\end{equation}
There is some $N$ such that every point in $\mathrm{supp}(x_N)$ has distance less than $\kappa$ from $\mathrm{supp}(y)$.
From $x_N$ we get a coarser $\RR$-filtration $x$ with support $Y$ and values
\begin{equation}
x_\pm(\lambda\pm\kappa)=x_{N,\pm}(\lambda\pm\kappa),\qquad\lambda\in Y
\end{equation}
which is paracomplemented by definition of $\kappa$ and satisfies $\mathrm{cl}(x)=y$.

We claim that $m$ takes only finitely many values on each fiber of $\mathrm{cl}$.
Indeed, if $a_0<\ldots<a_n$ is the HN filtration of $a\in\mathcal B(L)$, then $m(\Lambda(a))$ only depends on the partition 
\begin{equation}
\mathrm{cl}(a)=\ol{[a_0,a_1]}+\ldots+\ol{[a_{n-1},a_n]}
\end{equation}
of $\mathrm{cl}(a)$ into $0$-chains with positive coefficients, and there are only finitely many such partitions.
Taking fiberwise minimum of $m$ gives a function
\begin{equation}
f:\mathrm{cl}(V)\to\RR,\qquad f(x)=\min\{m(a)\mid \mathrm{cl}(a)=x\}.
\end{equation}
Since it has already established that $\mathrm{cl}(V)$ is compact, we can conclude that $m$ has a global minimum, and thus the existence of a weight filtration, if we show that $f$ is lower semicontinuous.

Let $x\in\mathrm{cl}(V)$, then $\rho=\rho(a)$ is the same for all $a$ with $\mathrm{cl}(a)=x$, since it only depends on the support.
After possibly shrinking $\rho$ we also have $\rho\leq \kappa(x)$, where $\kappa(x)=\kappa$ is defined as in \eqref{rad2} with $Y=\mathrm{supp}(x)$.
Let $O_\rho$ be the neighborhood of $x$ consisting of $0$-chains which differ from $x$ by a $1$-chain with support in a $\rho$-neighborhood of $\mathrm{supp}(x)$.
This is in complete analogy with the definition of $U_\rho(a)$ for $a\in\mathcal B(L)$, and we get
\begin{equation}
\mathrm{cl}^{-1}(O_\rho)=\bigcup_{\mathrm{cl}(a)=x}U_\rho(a)
\end{equation}
where the inclusion $\supseteq$ is clear and the inclusion $\subseteq$ follows from $\rho\leq\kappa$ by the same argument which showed that $\mathrm{cl}(V)$ is closed.

Suppose $b\in U_\rho(a)$ corresponds to an $\RR$-filtration $w$, then by the triangle inequality for mass, Theorem~\ref{mass_triangle_ineq}, and \eqref{nearby_lattice_split} we get
\begin{equation}\label{ma_upper_bound}
m(a)\leq \sum_{\lambda\in\RR}m([w_-(\lambda),w_+(\lambda)]).
\end{equation}
Let $c_{\lambda,0}<\ldots<c_{\lambda,n_\lambda}$ be the HN filtration in $[w_-(\lambda),w_+(\lambda)]$, then 
\begin{equation}
m([w_-(\lambda),w_+(\lambda)]) = \sum_{k=1}^{n_\lambda}\left|\sum_{\mu\in\RR}(1+\mu i)X([c_{\lambda,k-1,\mu},c_{\lambda,k,\mu}])\right|
\end{equation}
where $[w_-(\lambda),w_+(\lambda)]$ gets its polarization from $\Lambda(a)$ and
\begin{equation}\label{mb_actual}
m(b)=\sum_{\lambda\in\RR}\sum_{k=1}^{n_\lambda}\left|\sum_{\mu\in\RR}(1+(\lambda+\mu)i)X([c_{\lambda,k-1,\mu},c_{\lambda,k,\mu}])\right|.
\end{equation}
The difference between the right hand side of \eqref{ma_upper_bound} and \eqref{mb_actual} can be made smaller than some given $\varepsilon$ by suitable choice of $\rho$, which does not depend on the particular $a$ or $b$ but only a partition of $x$, of which there are finitely many.
This shows that $f$ is lower semicontinuous.
\end{proof}

Besides the weight filtration, any artinian lattice has two other canonically defined filtrations $0=a_0<a_1<\ldots<a_n=1$ such that the intervals $[a_{k-1},a_k]$ are complemented lattices.
The \textit{socle filtration} is defined inductively by the property that $a_k\in [a_{k-1},a_n]$ is maximal such that $[a_{k-1},a_k]$ is complemented. 
Dually, the \textit{cosocle filtration} is defined inductively by the property that $a_{k-1}\in [a_0,a_k]$ is minimal such that $[a_{k-1},a_k]$ is complemented.
Both are examples of a \textit{Loewy filtration}: A filtration of minimal length such that $[a_{k-1},a_k]$ are complemented lattices. 
These filtrations are typically considered in the context of representations of finite--dimensional algebras, see for example \cite{assem_et_al06}.

\subsection{Iterated weight filtration}
\label{sec_iterated}

If $(L,Z)$ is a semistable polarized lattice, then we can consider the subset $L'\subset L$ given by
\begin{equation}\label{semistable_sublattice}
L'=\{x\in L\mid x=0\text{ or }\phi([0,x])=\phi(L)\}
\end{equation}
which is a sublattice, hence artinian and there is a homomorphism 
\begin{equation}
X:K(L')^+\to\RR_{>0},\qquad X([x,y])=e^{-i\phi(L)}Z([x,y]).
\end{equation}
Moreover, $L'$ has strictly smaller length than $L$, unless the image of $Z$ is contained in a single ray.
If $L'$ is complemented, then $L$ is called \textit{polystable}.

We apply the above to the following situation.
Suppose $L$ is an artinian lattice with homomorphism $X:K^+(L)\to\RR_{>0}$ and let $a\in\mathcal B(L)$ be the weight filtration in $L$.
By definition, $\Lambda(a)$ is semistable, so we can consider $L^{(2)}=\Lambda(a)'$ which has a weight filtration $b$.
The filtration $b$ gives a filtration in $\Lambda(a)\supset L^{(2)}$, hence a refinement of $a$ to an $\RR^2$-filtration $a^{(2)}$ with
\begin{equation}
a^{(2)}_\pm(\lambda_1,\lambda_2)=b_\pm(\lambda_2)_{\lambda_1}\in [a_-(\lambda_1),a_+(\lambda_1)]
\end{equation}
where $\RR^2$ is given the lexicographical order.
By induction we get lattices $L^{(n)}$ and $\RR^n$-filtration $a^{(n)}$.
The lengths of $L^{(n)}$ are strictly decreasing until some $L^{(N+1)}$ is complemented and thus its weight filtration trivial, so the process stops after finitely many steps.
This shows that there is a canonical $\RR^\infty$-filtration in $L$, the \textbf{iterated weight filtration}, defined to be $a^{(N)}$.
We refer to $N$ as the \textbf{depth} of the iterated weight filtration.

\begin{figure}
\centering
\begin{tikzcd}
& & \bullet \arrow{drrr} & & & & & \\
& & & & & \bullet & & 
\end{tikzcd}

\vspace{\baselineskip}

\begin{tikzcd}
& & \bullet \arrow{dl}\arrow{drrr} & & & & \bullet \arrow{dl}  & \\
& \bullet & & & & \bullet & &
\end{tikzcd}

\vspace{\baselineskip}

\begin{tikzcd}
\bullet \arrow{dr} & & \bullet \arrow{dl}\arrow{dr}\arrow{drrr} & & \bullet \arrow{dr} & & \bullet \arrow{dl}\arrow{dr} \\
& \bullet & & \bullet & & \bullet & & \bullet
\end{tikzcd}

\vspace{\baselineskip}

\begin{tikzcd}
\bullet \arrow{d}\arrow{dr} & \bullet \arrow{d} & \bullet \arrow{dl}\arrow{d}\arrow{dr}\arrow{drrr} & \bullet \arrow{d} & \bullet \arrow{d}\arrow{dr} & \bullet \arrow{d} & \bullet \arrow{dl}\arrow{d}\arrow{dr} & \bullet \arrow{d} \\
\bullet & \bullet & \bullet & \bullet & \bullet & \bullet & \bullet & \bullet
\end{tikzcd}
\caption{The graphs $G^{(1)}$, $G^{(2)}$, $G^{(3)}$, $G^{(4)}$.}
\label{fig_itgraphs}
\end{figure}
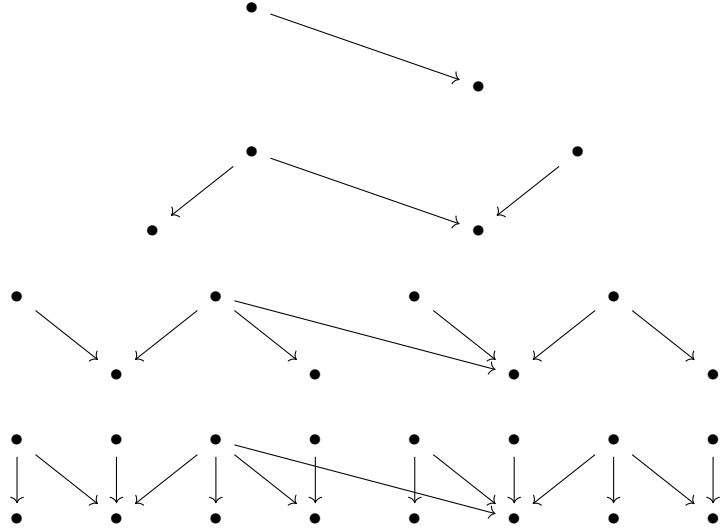

We will construct a series of examples generalizing the one in Section~\ref{sec_example} to show that the depth can be any non-negative integer.
The lattices will be obtained as lattices of closed subgraphs of oriented trees with the canonical homomorphism $X:K^+(L)\to\RR_{>0}$ given by the length of an interval.

Let $G^{(0)}$ be the graph with a single vertex and no edges and $G^{(1)}$ be the directed graph with two vertices and a single arrow between them.
Inductively define $G^{(n+1)}$ to be the directed graph obtained from $G^{(n)}$ by adding an outgoing arrow from each source to a new vertex and an incoming arrow to each sink starting at a new vertex. 
More formally, define vertices $G^{(n+1)}_0=G^{(n)}\times\{0,1\}$ and arrows $G^{(n+1)}_1$ to include $(i,0)\to (i,1)$ for each $i\in G^{(n)}_0$ and $(i,0)\to (j,1)$ for each arrow $i\to j$ in $G^{(n)}_1$
(see Figure~\ref{fig_itgraphs}).

The weight grading on $G^{(n)}$ is just $v_i=\frac{1}{2}$ if $i$ is a source and $v_i=-\frac{1}{2}$ if $i$ is a sink.
This follows from Lemma~\ref{weight_grading_lagr} with Lagrange multipliers $u_\alpha=\frac{1}{2}$ if $\alpha$ is a new arrow in $G^{(n)}$ and $u_\alpha=0$ otherwise.
To compute the iterated weight filtration we should next look at the lattice $L^{(2)}$ of closed subgraphs of $G^{(n)}$ such that the sum of $v_i$ is zero, i.e. which include an equal number of sinks and sources. 
It is easy to see that this coincides with the lattice of closed subgraphs of $G^{(n-1)}$. 
This shows that the iterated weight filtration on the lattice of closed subgraphs of $G^{(n)}$ has depth $n$.

\section{Gradient flow on quiver representations}
\label{sec_flow}

The purpose of this section is to show that the iterated weight filtration has a dynamical interpretation, describing the asymptotics of certain gradient flows which appear in the study of quiver representations.
We start by providing background on the K\"ahler geometry of spaces of quiver representations in the first subsection. 
Subection~\ref{sec_staralg} gives an alternative description of the flow in the language of $*$-algebras and $*$-bimodules, which is more invariant and simplifies formulas.
General properties of the flow are discussed in Subsection~\ref{sec_monotonicity}.
The final subsection completes the proof of our second main theorem by giving a construction of asymptotic solutions.

\subsection{K\"ahler geometry of quiver representations}
\label{sec_kahler}

Many problems in linear algebra are instances of the following general one.
Given a quiver
\begin{equation}
\begin{tikzcd}
 & Q_1 \arrow{dl}[swap]{s} \arrow{dr}{t} & \\
Q_0 & & Q_0 \\
\end{tikzcd}
\end{equation}
where $Q_0$ is the set of vertices, $Q_1$ the set of arrows, and $s$ and $t$ assign to each arrow its starting and target vertex, classify all the ways in which such a diagram can be realized (represented) using finite-dimensional vector spaces and linear maps.
The space of representations for fixed vector spaces $E_i$, $i\in Q_0$, is a quotient
\begin{equation}
\bigoplus_{\alpha:i\to j}\mathrm{Hom}(E_i,E_j){\bigg{/}}\prod_{i\in Q_0}GL(E_i)=:V/G
\end{equation}
of a vector space by a reductive group.

If the ground field is $\CC$ then $V/G$ is approximated by a K\"ahler manifold.
To construct it, choose a Hermitian metric on each $E_i$, then the norm-squared
\begin{equation}
S(\phi)=\sum_{\alpha:i\to j} \tr\left(\phi_\alpha^*\phi_\alpha\right)
\end{equation}
where $\phi_\alpha\in \mathrm{Hom}(E_i,E_j)$, is a K\"ahler potential for the flat metric on $V$.
We can look for points in $V$ which minimize $S$ on a given $G$-orbit. 
These are representations with
\begin{equation}
\sum_{\alpha:i\to j}[\phi_\alpha^*,\phi_\alpha]=0.
\end{equation}
Such a minimum can be found if and only if the $G$-orbit corresponds to a semisimple representation. 
This is an application of the Kempf--Ness theorem.
A K\"ahler manifold is then obtained as the quotient of the set of minimizers by the unitary subgroup $K\subset G$ preserving the metric on each $E_i$, with the potential which is the restriction of $S$.

If $Q$ has no oriented cycles then the only semisimple representations are those with $\phi=0$.
Following ideas from geometric invariant theory A.~King~\cite{king94} shows how to obtain non-trivial spaces by generalizing the above construction. 
They depend on a choice of \textit{polarization}, which is in this context just a real number $\theta_i\in\RR$ for each vertex $i\in Q_0$.
They allow us to extend the action of $G$ to $V\times\CC$ by letting $g\in G$ act on $z\in\CC$ by multiplication with
\begin{equation}
\prod_{i\in G_0}\left(\det g_i\right)^{\theta_i}.
\end{equation}
(Strictly speaking, this is ill-defined if $\theta_i$ are not integers and we should work with virtual line bundles.)  
On $V\times \CC^*$ consider the potential
\begin{equation}
S(\phi,z)=\sum_{\alpha\in Q_1} \tr\left(\phi_\alpha^*\phi_\alpha\right)+\log|z|.
\end{equation}
Fixing $\phi\in V$, we can consider the ($K$-invariant) restriction of $S$ to the orbit $G(\phi,1)$ as a function on the homogeneous space $G/K$. 
A point in $G/K$ corresponds to a choice of positive definite Hermitian endomorphism $h_i$ on each $E_i$, and
\begin{equation}
S(h)=\sum_{\alpha:i\to j} \tr\left(h_i^{-1}\phi_\alpha^*h_j\phi_\alpha\right)+\sum_{i\in Q_0}\theta_i\log\det h_i.
\end{equation}
The equation for $h\in G/K$ to be a critical point of $S$ is
\begin{equation}\label{rep_harmonic_metric}
\sum_{\alpha:i\to j}[h_i^{-1}\phi_\alpha^*h_j,\phi_\alpha]=\sum_{i\in Q_0}\theta_i\mathrm{pr}_{E_i}.
\end{equation}

To describe those representations for which the above equation has a solution, we need to recall some terminology.
For any representation $E$ of $Q$ we define
\begin{equation}
\theta(E)=\sum_{i\in Q_0}\theta_i\dim E_i
\end{equation}
and say that $E$ is \textbf{semistable} if $\theta(E)=0$ and any subrepresentation $F\subset E$ satisfies $\theta(F)\leq 0$.
If in addition $\theta(F)<0$ whenever $0\neq F\subsetneq E$, then $E$ is called \textbf{stable}.
Finally, $E$ is \textbf{polystable} if it is a direct sum of stable representations.
Note that for $\theta=0$ all representations are semistable and polystable$=$semisimple.

\begin{thm}[King]\label{king_thm}
$S$ is bounded below on the $G$-orbit through $(\phi,z)\in V\times\CC$ if and only if $\phi$ defines a semistable representation.
Moreover, there is a solution to \eqref{rep_harmonic_metric}, i.e. a minimum point of $S$, if and only if  the representation is polystable. 
\end{thm}

The set of polystable representations (up to isomorphism) thus has the structure of a K\"ahler manifold. (More precisely a stratified K\"ahler manifold, see \cite{zbMATH00038411}.)

From a dynamical point of view, polystability means that the gradient flow of $S$ on $G/K$ has the simplest possible asymptotics: exponentially fast convergence to a fixed point, which is a solution of \eqref{rep_harmonic_metric}.
One can study the asymptotic behavior of the flow for non-polystable representations and see if this yields more information about $V/G$. 

To define a gradient of $S$ we need to choose a Riemannian metric on $G/K$.
We consider metrics of the form
\begin{equation}
\langle v,w\rangle=\sum_{i\in Q_0}m_i\mathrm{tr}\left(h_i^{-1}vh_i^{-1}w\right), \qquad v,w\in T_h(G/K)
\end{equation}
where $m_i>0$, $i\in Q_0$, are some fixed positive numbers.
The negative gradient flow is then 
\begin{equation}\label{S_gradient_flow_long}
m_ih_i^{-1}\frac{dh_i}{dt}=\sum_{\alpha:i\to j}h_i^{-1}\phi_\alpha^*h_j\phi_\alpha-\sum_{\alpha:j\to i}\phi_\alpha h_j^{-1}\phi_\alpha^*h_i-\theta_i.
\end{equation}
We will show in this section that in the semistable case the asymptotics of this flow are completely described by the iterated weight filtration.
More precisely, on the $E_\lambda$ piece of the filtration, $\lambda=(\lambda_1,\ldots,\lambda_n)$, we have
\begin{equation}
\log(|h(t)|)=\lambda_1\log t+\lambda_2\log\log t+\ldots+\lambda_n\log^{(n)}t+O(1).
\end{equation}

\subsection{Star-algebras and bimodules}
\label{sec_staralg}

In order to simplify formulas like \eqref{S_gradient_flow_long} and all calculations below, it is useful to adopt the more invariant language of $*$-algebras and $*$-bimodules.
This offers perhaps also a more algebraic point of view on the K\"ahler geometry discussed in the previous subsection.
To motivate the general definitions below, we first describe the structure in the case of quiver representations.

To begin, note that
\begin{equation}
B:=\prod_{i\in Q_0}\mathrm{End}(E_i)
\end{equation}
is a finite-dimensional $C^*$-algebra, with $*$-structure determined by the choice of metrics on the vector spaces $E_i$.
It follows from the classification of type I factors, or more directly using the Artin--Wedderburn theorem, that every finite-dimensional $C^*$-algebra is of this form.
Finite-dimensional $C^*$-algebras are also precisely those $*$-algebras which have a faithful finite-dimensional $*$-representation on an inner product space.
Recall that a $*$-algebra over $\CC$ is a $\CC$-algebra, $A$, together with a map $A\to A$, $a\mapsto a^*$ such that
\begin{equation}
a^{**}=a,\qquad (a+b)^*=a^*+b^*,\qquad (\lambda a)^*=\overline{\lambda}a^*,\qquad (ab)^*=b^*a^*
\end{equation}
for $a,b\in A$, $\lambda\in\CC$.

The masses $m_i>0$, $i\in Q_0$, determine a positive trace
\begin{equation}
\tau:B\to\mathbb C,\qquad b\mapsto\sum_{i\in Q_0}m_i\mathrm{tr}(b_i).
\end{equation}
Functionals $\tau$ obtained in such a way are characterized by the properties
\begin{equation}
\tau(a^*)=\overline{\tau(a)},\qquad \tau(ab)=\tau(ba),\qquad \tau(aa^*)>0 \text{ for } a\neq 0
\end{equation}
which imply that
\begin{equation}
\langle a,b\rangle:=\tau(a^*b)
\end{equation}
defines a Hermitian inner product on $A$.

The space of representations
\begin{equation}
M:=\bigoplus_{\alpha:i\to j}\mathrm{Hom}(E_i,E_j)
\end{equation}
has the structure of a $B$--$B$ bimodule.
Additionally, there are two $B$-valued inner products
\begin{gather}
(\phi,\psi)\mapsto \phi\psi^*=\sum_{\alpha:i\to j}\frac{1}{m_j}\phi_\alpha\psi_\alpha^* \\
(\phi,\psi)\mapsto \phi^*\psi=\sum_{\alpha:i\to j}\frac{1}{m_i}\phi_\alpha^*\psi_\alpha
\end{gather}
where the normalization is chosen so that
\begin{equation}
\tau(\phi\psi^*)=\tau(\phi^*\psi).
\end{equation}

More generally, suppose $A,B$ are finite dimensional $C^*$-algebras with trace. 
If $M$ is an $A$--$B$ bimodule, then $\overline{M}$ is the complex conjugate vector space with identity map $M\to\overline{M}$, $m\mapsto m^*$ and $B$--$A$ bimodule structure given by
\begin{equation}
bm^*a:=(a^*mb^*)^*
\end{equation}
for $a\in A$, $b\in B$, $m\in M$. 
We say $M$ is a \textbf{$*$-bimodule} if it is equipped with homomorphisms of bimodules
\begin{gather}
M\otimes_B \overline{M}\to A,\qquad m\otimes n^*\mapsto mn^* \\
\overline{M}\otimes_A M\to B,\qquad m^*\otimes n\mapsto m^*n
\end{gather}
which are algebra-valued inner products on $M$ in the sense that
\begin{gather}
(mn^*)^*=nm^*,\qquad mm^*\geq 0, \qquad mm^*=0\implies m=0 \\
(m^*n)^*=n^*m,\qquad m^*m\geq 0, \qquad m^*m=0\implies m=0
\end{gather}
and are related by
\begin{equation}\label{bimodule_trace_rel}
\tau_A(mn^*)=\tau_B(n^*m).
\end{equation}
Caution: In general one has $(ab^*)c\neq a(b^*c)$ for $a,b,c\in M$.

The finite-dimensional $B$--$B$ $*$-bimodules are, up to isomorphism, all obtained as above from quivers.
The following table summarizes our setup and the dictionary between the two languages.

\begin{center}
\begin{tabular}{|c|c|c|}
\hline
notation & type & in terms of quiver $Q$ \\
\hline
\hline
$B$ & finite-dim. $C^*$-algebra & $\bigoplus_{i\in Q_0}\mathrm{End}(E_i)$ \\
\hline
$\tau$ & trace $B\to\CC$ & $\sum_{i\in Q_0}m_i\tr(b_i)$, $m_i>0$ \\
\hline
$\rho$ & $\rho\in\mathrm{center}(B)$, $\rho=\rho^*$ & $\rho_i=\theta_i/m_i\in\RR$, $i\in Q_0$ \\
\hline
$M$ & $B$--$B$ $*$-bimodule & $\bigoplus_{\alpha:i\to j}\mathrm{Hom}(E_i,E_j)$ \\
\hline
$\phi$ & element of $M$ & $\phi_\alpha:E_i\to E_j$, $\alpha:i\to j$ \\
\hline
\end{tabular}
\end{center}

For example the equation \eqref{S_gradient_flow_long} for the flow now takes the form
\begin{equation}\label{gradflow}
h^{-1}\frac{dh}{dt}=[h^{-1}\phi^* h,\phi]-\rho
\end{equation}
where $h\in B$ moves in the cone 
\begin{equation}
\mathcal P:=\{h\in B\mid h^*=h, \mathrm{Spec}(h)\in(0,\infty)\}\subset B
\end{equation}
of self-adjoint operators with strictly positive spectrum (which was written as $G/K$ before).

For the remainder of this subsection we show how to obtain from a triple $(B,\tau,M)$ a new one $(B',\tau',M')$ by deforming (``twisting'') along an element $\phi\in M$.

\begin{lem}\label{diff_adjoint}
Let $\phi\in M$, then the adjoint of $[\phi,\_]:B\to M$ is $[\phi^*,\_]:M\to B$
\end{lem}

\begin{proof}
\begin{align}
\langle[\phi,b],m\rangle &= \tau([b^*,\phi^*]m) \\
                         &= \tau(b^*\phi^*m-\phi^*b^*m) \\
                         &= \tau(b^*[\phi^*,m]) \\
                         &= \langle b,[\phi^*,m]\rangle
\end{align}
Note the use of \eqref{bimodule_trace_rel}.
\end{proof}

The kernel of $[\phi,\_]$ is a subalgebra in general, but it need not be closed under the operation $*$.
The following proposition states that, under the condition of centrality of $[\phi^*,\phi]$, passing to the ``harmonic part'' of the complex $B\to M$ produces another pair $(B',M')$ of the same sort.

\begin{prop}\label{harm_part_prop}
Suppose that 
\begin{equation}
[\phi^*,\phi]\in\mathrm{center}(B)
\end{equation}
and let
\begin{equation}
B'=\{b\in B\mid [\phi,b]=0\},\qquad M'=\{m\in M\mid [\phi^*,m]=0\}
\end{equation}
then $B'$ is a $*$-subalgebra of $B$ and $M'$ a $B'$--$B'$ $*$-bimodule with $B'$-valued inner products given by composition of those of $M$ with the orthogonal projection $B\to B'$.
\end{prop}

\begin{proof}
We have
\begin{align}
\langle [\phi,b],[\phi,b]\rangle &= \langle b,[\phi^*,[\phi,b]]\rangle \\
                                 &= \langle b,[\phi,[\phi^*,b]]\rangle \\
                                 &= \langle [\phi^*,b],[\phi^*,b] \rangle
\end{align}
but $[\phi^*,b]=-[\phi,b^*]^*$, so $[\phi,b]=0$ implies $[\phi,b^*]=0$.
Thus, $B'$ is a $*$-subalgebra.

If $b\in B'$, $m\in M'$, then
\begin{equation}
[\phi^*,bm]=[\phi^*,b]m+b[\phi^*,m]=0
\end{equation}
so $bm\in M'$ and similarly $mb\in M'$. 
Thus $M'$ is a $B'$--$B'$ bimodule.

Next, let $P:B\to B'$ be the orthogonal projection.
By Lemma~\ref{diff_adjoint} it is characterized by $P(b)\in B'$ and $P(b)-b=[\phi^*,m]$ for some $m\in M$.
This also shows that $\tau(P(b))=\tau(b)$.
We claim that if $a\in B'$, $b\in B$, then $P(ab)=aP(b)$.
To see this, let $P(b)-b=[\phi^*,m]$, then
\begin{equation}
aP(b)-ab=a[\phi^*,m]=[\phi^*,am].
\end{equation}
As a consequence, we see that 
\begin{gather}
M'\otimes \overline{M'}\to B',\qquad m\otimes n^*\mapsto P(mn^*) \\
\overline{M'}\otimes M'\to B',\qquad m^*\otimes n\mapsto P(m^*n)
\end{gather}
are maps of bimodules. 
Also, \eqref{bimodule_trace_rel} for $M'$ follows from the corresponding identity for $M$ and $\tau(P(b))=\tau(b)$.
\end{proof}

\subsection{Monotonicity and homogeneity}
\label{sec_monotonicity}

A key property of the flow \eqref{gradflow}, for our purposes, is a certain kind of monotonicity.

\begin{prop}[Monotonicity]
\label{prop_mono}
Let $h_1(t),h_2(t)$ be solutions of \eqref{gradflow} with $h_1(0)\leq h_2(0)$, then $h_1(t)\leq h_2(t)$ for all $t\geq 0$.
\end{prop}

\begin{proof}
Consider
\begin{equation}
A:=\{(g,h)\in\mathcal P\times\mathcal P\mid g\leq h\}
\end{equation}
which is a manifold with corners.
To prove the proposition it suffices to show that the flow on pairs $(g,h)\in\mathcal P\times\mathcal P$ is pointing inwards or in a tangential direction on the boundary $\partial A$, which is the subset where $g-h$ is not invertible.
Assume, for convenience, that $B$ is given concretely as
\begin{equation}
B=\mathrm{End}(V_1)\times\ldots\times\mathrm{End}(V_n)
\end{equation}
where $V_i$ are finite-dimensional Hermitian spaces.
Then the claim to check is that 
\begin{equation}
b:=h\left[h^{-1}\phi^*h,\phi\right]-g\left[g^{-1}\phi^*g,\phi\right]-(h-g)\rho
\end{equation}
is non-negative on $\mathrm{Ker}(g-h)$ for $(g,h)\in\partial A$.

Since the flow is coordinate-independent, we may assume that $h$ is the identity.
So let $v\in V:=\bigoplus V_i$ with $g(v)=v$, then
\begin{align}
v^*bv&=v^*\left(\phi^*\phi-\phi\phi^*-\phi^*g\phi+g\phi g^{-1}\phi^*g\right)v \\
&=(\phi v)^*(1-g)\phi v+(\phi^*v)^*\left(g^{-1}-1\right)\phi^*v\geq 0
\end{align}
since $1-g\geq 0$ and thus $g^{-1}-1\geq 0$.
\end{proof}

As a first consequence we see that any two solutions have the same asymptotics by a ``sandwiching'' argument.

\begin{coro}
Let $h_1,h_2$ be solutions of \eqref{gradflow} for $t\geq 0$. 
Then there is a constant $C>0$ such that
\begin{equation}
\frac{1}{C}h_1(t)\leq h_2(t)\leq Ch_1(t)
\end{equation}
for $t\geq 0$.
\end{coro}

A related result is established by Harada--Wilkin~\cite{harada_wilkin} who show that the flow is distance decreasing.

\begin{proof}
We can find a $C>0$ such that the inequality holds for $t=0$. 
By monotonicity, it holds for all $t\geq 0$.
\end{proof}

Call $h$ an \textbf{asymptotic solution} of \eqref{gradflow} if for some (hence any) actual solution $g$ there is a $C>0$ such that $C^{-1}g(t)\leq h(t)\leq Cg(t)$ for sufficiently large $t$.
We will find that \eqref{gradflow} always admits explicit asymptotic solutions in terms of iterated logarithms, and these are generally not actual solutions.

We return to the point of view that the flow is changing coordinates on the $E_i$'s instead of the metric.
Write $h=x^*x$, then \eqref{gradflow} implies that
\begin{equation}\label{gradflow_root}
\dot{x}x^{-1}+\left(\dot{x}x^{-1}\right)^*=\left[\left(x\phi x^{-1}\right)^*,x\phi x^{-1}\right]-\rho.
\end{equation}
Note that this equation only determines the selfadjoint part of $\dot{x}x^{-1}$, which corresponds to the fact that $x$ is  determined only up to multiplication by unitary elements on the left.

\begin{prop}[Homogeneity]
\label{prop_homogen}
Let $x$ be a solution of \eqref{gradflow_root} and $f:\RR\to\RR$ a continuous function, then
\begin{equation}
y:=x\exp\left(\frac{1}{2}\int f\right)
\end{equation}
solves
\begin{equation}
\dot{y}y^{-1}+\left(\dot{y}y^{-1}\right)^*=\left[\left(y\phi y^{-1}\right)^*,y\phi y^{-1}\right]-\rho+f.
\end{equation}
\end{prop}

\begin{proof}
Let 
\begin{equation}
F=\frac{1}{2}\int f
\end{equation}
then
\begin{equation}
\dot{y}y^{-1}=\left(\dot{x}e^F+x\dot{F}e^F\right)e^{-F}x^{-1}=\dot{x}x^{-1}+\dot{F}
\end{equation}
and the right hand side of \eqref{gradflow_root} remains unchanged if $x$ is replaced by $y$.
\end{proof}

The following gives a sufficient criterion to recognize asymptotic solutions. 
It relies on monotonicity and homogeneity.

\begin{prop}\label{sol_up_to_l1}
Suppose
\begin{equation}
\dot{x}x^{-1}+\left(\dot{x}x^{-1}\right)^*=\left[\left(x\phi x^{-1}\right)^*,x\phi x^{-1}\right]-\rho+s
\end{equation}
with $s=s(t)$ an absolutely integrable function with values in selfadjoint elements of $B$. 
Then $h=x^*x$ is an asymptotic solution of \eqref{gradflow}.
\end{prop}

\begin{proof}
In the special case when $s$ is scalar-valued (i.e. takes values in $\RR\cdot 1$) the claim follows immediately from Proposition~\ref{prop_homogen}, since the absolute value of an antiderivative of $s$ is bounded by assumption.
For the general case it suffices to show (by symmetry) that if $f$ is scalar-valued with $f\geq s$ and $y$ solves
\begin{equation}
\dot{y}y^{-1}+\left(\dot{y}y^{-1}\right)^*=\left[\left(y\phi y^{-1}\right)^*,y\phi y^{-1}\right]-\rho+f
\end{equation}
with $(x^*x)(0)\leq (y^*y)(0)$, then $(x^*x)(t)\leq (y^*y)(t)$ for $t\geq 0$.
This is a strengthening of the monotonicity property, and proven in much the same way as Proposition~\ref{prop_mono}.
The only modification is the following:
Assuming $x(0)=1$ after a change of coordinates, the additional term is
\begin{equation}
v^*(y(0))^*f(0)y(0)v-v^*s(0)v=v^*(f(0)-s(0))v\geq 0,
\end{equation}
where we use the fact that $f(0)$ is a scalar and $(y^*y)(0)v=v$.
\end{proof}

\subsection{Asymptotic solution}
\label{sec_ansatz}

The goal of this subsection is to construct an asymptotic solution of \eqref{gradflow} using the iterated weight filtration on the lattice of subrepresentations, which is described in terms of our $*$-data as follows.
The selfadjoint elements of $B$ are partially ordered by $a\leq b$ iff $b-a$ is a non-negative operator.
In particular, we get a partial order on \textit{projectors}, those $p\in B$ with $p^2=p^*=p$.
Because $B$ is finite-dimensional, the poset $\Lambda(B)$ of projectors is an artinian modular lattice.
To see this, identify $B$ with a product of matrix algebras and projectors with their images.
The lattice of subrepresentations of $\phi\in M$ is the sublattice
\begin{equation}
\Lambda(B,\phi):=\{p\in\Lambda(B)\mid p\phi p=\phi p \}
\end{equation}
of projectors which are compatible with $\phi$.
The trace $\tau$ on $B$ together with $\rho\in B$ provide a polarization
\begin{equation}
Z:K(\Lambda(B,\phi))\to\CC,\qquad [p,q]\mapsto \tau((1+\rho i)(q-p))
\end{equation}
which sends positive classes to the right half-plane.
By the general theory, $\Lambda(B,\phi)$ has a HN filtration, and each semistable interval is further refined by a balanced filtration, perhaps iterated.
Since we are mainly interested in the refinement of the HN filtration, we assume that 
\begin{equation}\label{assert_semistable}
\Lambda(B,\phi)  \text{ is semistable of phase } 0
\end{equation}
which means that $\tau(\rho)=0$ and $\tau(\rho p)\leq 0$ for all $p\in\Lambda(B,\phi)$, i.e. just semistability with respect to $\theta$.

By our general theory for modular lattices there is a canonical iterated weight filtration $p_\pm(\lambda)$ in the sublattice $\Lambda(B,\phi)^0$ of semistables of phase $0$, as in \eqref{semistable_sublattice}. 
Let 
\begin{equation}
p_\lambda:=p_+(\lambda)-p_-(\lambda)
\end{equation}
which is a projector, though usually not in $\Lambda(B,\phi)$.
Since $p_\pm(\lambda)$ is an $\RR$-filtration, the $p_\lambda$ are mutually orthogonal and sum to $1$.
Split $\phi$ into its $\lambda$-components
\begin{equation}
\phi_\lambda=\sum_{\mu\in\RR}p_{\mu+\lambda}\phi p_\mu
\end{equation}
then $\phi_\lambda=0$ for $\lambda>0$ since $p_\pm(\lambda)\in\Lambda(B,\phi)$.

Since each interval
\begin{equation}
\left[p_\pm(\lambda),p_\pm(\lambda+1)\right]\subset\Lambda(B,\phi)^0
\end{equation}
is complemented by assumption, we can ensure, after conjugating $\phi$ by a suitable invertible element $b\in B$, that
\begin{equation}\label{phi_gap}
\phi_\lambda=0,\qquad\lambda\in(-1,0)
\end{equation}
where $b$ is chosen to take a splitting to an orthogonal one.
Also, by definition of $\Lambda(B,\phi)^0$ and the assumption that each $[p_+(\lambda),p_-(\lambda)]$ is complemented, applying Theorem~\ref{king_thm} and further conjugating $\phi$ we have
\begin{equation}\label{phi0_harm}
\left[\phi_0^*,\phi_0\right]=\rho.
\end{equation}
Furthermore, we can choose harmonic representatives of the $\phi_\lambda$, $\lambda\leq-1$, meaning we conjugate $\phi$ to get
\begin{equation}\label{phi1_harm}
\left[\phi_0^*,\phi_\lambda\right]=0,\qquad \lambda\leq-1.
\end{equation}

Let
\begin{equation}
r=\sum_{\lambda\in\RR}\lambda p_\lambda
\end{equation}
and
\begin{gather}
B':=\{b\in B\mid [r,b]=0,[\phi_0,b]=0\} \\
M':=\{m\in M\mid [r,m]=-m,[\phi_0^*,m]=0\}
\end{gather}
which are the harmonic $r$-degree $0$ part of $B$ and harmonic $r$-degree $-1$ part of $M$ respectively.
A slight extension of the proof of Proposition~\ref{harm_part_prop} shows that $B'$ is a $*$-subalgebra and $M'$ is a $B'$--$B'$ $*$-bimodule.
The point is that if $m,n\in M'$ then $[r,m^*n]=0$ automatically, but we still need to project to the harmonic part to get an element of $B'$ as in Proposition~\ref{harm_part_prop}.
Note also that 
\begin{equation}
\rho':=r\in \mathrm{center}(B'),\qquad \phi':=\phi_{-1}\in M'
\end{equation}
by definition and \eqref{phi1_harm}.
The defining property of the iterated weight filtration guarantees that the new quadruple $(B',\rho',M',\phi')$ is semistable of phase $0$.

Let $x$ be a solution of the flow \eqref{gradflow_root} for $(B',\theta',M',\phi')$, so
\begin{equation}\label{derived_quad_eq}
\dot{x}x^{-1}+\left(\dot{x}x^{-1}\right)^*=P\left(\left[\left(x\phi_{-1}x^{-1}\right)^*,x\phi_{-1}x^{-1}\right]\right)-r
\end{equation}
where $P:B\to B'$ is the orthogonal projection.
Note that $(B',\theta',M',\phi')$ is \textit{polystable}, i.e. $\Lambda(B',\phi')^0$ complemented, if and only if there exists a constant solution $x$.
It follows from the calculation below and induction that $x$ grows at most polynomially in general.

\begin{lem}\label{lem_log_ansatz}
Suppose $x:[0,\infty)\to B'$ is a solution to \eqref{derived_quad_eq}, then
\begin{equation}
y(t):=t^{r/2}x(\log t)
\end{equation}
satisfies
\begin{equation}
\dot{y}y^{-1}+\left(\dot{y}y^{-1}\right)^*=P\left(\left[\left(y\phi_{-1}y^{-1}\right)^*,y\phi_{-1}y^{-1}\right]\right).
\end{equation}
\end{lem}

\begin{proof}
Indeed,
\begin{align}
\frac{dy}{dt}(t)(y(t))^{-1} &= \left(\frac{r}{2}t^{r/2-1}x(\log t)+t^{r/2}\frac{dx}{dt}(\log t)t^{-1}\right)(x(\log t))^{-1}t^{-r/2} \\
&= t^{-1}\left(\frac{r}{2}+\frac{dx}{dt}(\log t)(x(\log t))^{-1}\right)
\end{align}
and
\begin{equation}
y(t)\phi_{-1}(y(t))^{-1}=t^{-1/2}x(\log t)\phi_{-1}(x(\log t))^{-1}
\end{equation}
since $\phi_{-1}$ has $r$-degree $-1$.
We use here the fact that $x$ commutes with $r$, thus $t^{r/2}$, as $x(t)\in B'$ by definition.
\end{proof}

Let
\begin{equation}
\Delta:B\to B,\qquad \Delta(b)=\left[\phi_0^*,[\phi_0,b]\right]
\end{equation}
and
\begin{equation}
G:B\to B,\qquad 1=P+\Delta G=P+G\Delta,\qquad PG=GP=0
\end{equation}
the ``Green's operator''.
All three $P,\Delta,G$ are endomorphisms of $B$ as a $B'$--$B'$ bimodule and commute with the $*$ operation.

Let $x,y$ be as in the previous lemma and consider
\begin{gather}
k:=\left[\left(y\phi_{-1}y^{-1}\right)^*,y\phi_{-1}y^{-1}\right] \\
z:=y\left(1+\frac{1}{2}G(y^{-1}ky)\right).
\end{gather}

\begin{lem}
The function $z$ above is a solution of \eqref{gradflow_root} up to terms in $L^1$, i.e. satisfies the hypothesis of Proposition~\ref{sol_up_to_l1} and is thus an asymptotic solution.
\end{lem}

It is important to note that the factor $y^{-1}z=(1+\frac{1}{2}G(y^{-1}ky))$ is bounded for large $t$, hence does not change the asymptotics. 
It is only needed to get a solution up to terms in $L^1$.

\begin{proof}
We write $O(t^\alpha\mathcal L)$ for terms which are $O(t^{\alpha})$ up to logarithmic corrections, e.g. $O(t^{\alpha}\log t)$, $O(t^\alpha\log t\log\log t)$, and so on.
For instance, since $\phi_{-1}$ has $r$-degree $-1$ and $k$ has $r$-degree $0$ we have
\begin{equation}
y\phi_{-1}y^{-1}=O(t^{-1/2}\mathcal L),\qquad k=O(t^{-1}\mathcal L),\qquad G(y^{-1}ky)=O(t^{-1}\mathcal L).
\end{equation}
Consequently,
\begin{equation}
\left(1+\frac{1}{2}G(y^{-1}ky)\right)^{-1}=\left(1-\frac{1}{2}G(y^{-1}ky)\right)+O(t^{-2}\mathcal L)
\end{equation}
and
\begin{equation}
\dot{z}=\dot{y}\left(1+\frac{1}{2}G(y^{-1}ky)\right)+yO(t^{-2}\mathcal L)
\end{equation}
where the terms in $O(t^{-2}\mathcal L)$ are of $r$-degree $0$, hence
\begin{equation}
\dot{z}z^{-1}=\dot{y}y^{-1}+O(t^{-2}\mathcal L).
\end{equation}

Recall the splitting
\begin{equation}
\phi=\phi_0+\phi_{-1}+\phi_{<-1}
\end{equation}
where $\phi_{<-1}$ collects components of $r$-degree $<-1-\epsilon$ for some $\epsilon>0$.
We have
\begin{align}
z\phi_0z&=y\left(1+\frac{1}{2}G(y^{-1}ky)\right)\phi_0\left(1-\frac{1}{2}G(y^{-1}ky)\right)y^{-1}+O(t^{-2}\mathcal L) \\
&=\phi_0+\frac{1}{2}\left[G(k),\phi_0\right]+O(t^{-2}\mathcal L)
\end{align}
and
\begin{gather}
z\phi_{-1}z^{-1}=y\phi_{-1}y^{-1}+O(t^{-3/2}\mathcal L) \\
z\phi_{<-1}z^{-1}=O(t^{(-1-\epsilon)/2}\mathcal L)
\end{gather}
hence
\begin{align}
[(z\phi_0z^{-1})^*&,z\phi_0z^{-1}]= \\
&=[\phi_0^*,\phi_0]+\frac{1}{2}[\phi_0^*,[G(k),\phi_0]]+\frac{1}{2}[[G(k),\phi_0]^*,\phi_0]+O(t^{-2}\mathcal L)\\
&=\theta-\Delta G(k)+O(t^{-2}\mathcal L) \\
&=\theta+(P-1)(k)+O(t^{-2}\mathcal L)
\end{align}
by \eqref{phi0_harm}.
Furthermore, by \eqref{phi1_harm},
\begin{gather}
\left[(z\phi_0z^{-1})^*,z\phi_{-1}z^{-1}\right]=O(t^{-3/2}\mathcal L) \\
\left[(z\phi_{-1}z^{-1})^*,z\phi_0z^{-1}\right]=O(t^{-3/2}\mathcal L)
\end{gather}
and
\begin{equation}
\left[(z\phi_{-1}z^{-1})^*,z\phi_{-1}z^{-1}\right]=\left[(y\phi_{-1}y^{-1})^*,y\phi_{-1}y^{-1}\right]+O(t^{-2}\mathcal L).
\end{equation}
Finally, combining the above we get
\begin{equation}
\left[(z\phi z^{-1})^*,z\phi z^{-1}\right]=\theta+P(k)+O(t^{-1-\epsilon}\mathcal L)
\end{equation}
thus
\begin{align}
\dot{z}z^{-1}+\left(\dot{z}z^{-1}\right)^* &= \dot{y}y^{-1}+\left(\dot{y}y^{-1}\right)^*+O(t^{-2}\mathcal L) \\
&=P(k)+O(t^{-2}\mathcal L) \\
&=\left[(z\phi z^{-1})^*,z\phi z^{-1}\right]-\theta+O(t^{-1-\epsilon}\mathcal L)
\end{align}
which completes the proof.
\end{proof}

Let 
\begin{equation}
1=\sum_{\lambda\in\RR^{\infty}}p_\lambda
\end{equation}
be the orthogonal splitting of the identity in $B$ given by the iterated weight filtration.
Disregarding multiplicatively bounded terms coming from the Green's operator, the asymptotic solution of \eqref{gradflow} constructed in the proof above is
\begin{equation}\label{asymp_sol_final}
\sum_{\lambda=(\lambda_1,\ldots,\lambda_n)\in\RR^{\infty}} t^{\lambda_1}(\log t)^{\lambda_2}\cdots \left(\log^{(n-1)}t\right)^{\lambda_n}p_\lambda
\end{equation}
where $\log^{(k)}$ is the $k$-times iterated logarithm.

\begin{coro}
After conjugating $\phi$ by a suitable invertible element in $B$, the function $[C,\infty)\to B$ given by \eqref{asymp_sol_final} is an asymptotic solution of \eqref{gradflow}.
\end{coro}

\begin{thm}\label{thm_asymptotics}
Suppose $B$, $\tau$, $\theta$, $M$, $\phi$ are as in Section~\ref{sec_staralg} and $h$ a solution of \eqref{gradflow}.
For $\lambda=(\lambda_1,\ldots,\lambda_n)\in\RR^{\infty}$ let $p_\pm(\lambda)\in\Lambda(B,\phi)$ be the projector which is the $\lambda$-step of the iterated weight filtration in the polarized lattice $\Lambda(B,\phi)$, then
\begin{equation}
p_+(\lambda)hp_+(\lambda)=O\left(t^{\lambda_1}(\log t)^{\lambda_2}\cdots \left(\log^{(n-1)}t\right)^{\lambda_n}\right)
\end{equation}
and
\begin{equation}
p_-(\lambda)hp_-(\lambda)=o\left(t^{\lambda_1}(\log t)^{\lambda_2}\cdots \left(\log^{(n-1)}t\right)^{\lambda_n}\right).
\end{equation}
\end{thm}

\begin{ex}
The simplest non-trivial example is the representation
\begin{equation}
\CC\xrightarrow{1/\sqrt{2}}\CC
\end{equation}
of the $A_2$ quiver. 
The normalization is chosen so that \eqref{phi0_harm} holds at the second level.
Also, we set $\theta=0$ and $\tau=\tr$.
The equations \eqref{gradflow} are
\begin{equation}
\dot{h_1}/h_1=\frac{1}{2}h_2/h_1, \qquad \dot{h_2}/h_2=-\frac{1}{2}h_2/h_1
\end{equation}
with asymptotic solution
\begin{align}
h_1=t^{1/2},\qquad h_2=t^{-1/2}
\end{align}
which happens to be an exact solution.
\end{ex}

\begin{ex}
Let us look at an example where the weight filtration is iterated.
Namely, take the representation
\begin{equation}
\CC\xrightarrow{1/\sqrt{2}}\CC\xleftarrow{\quad 1\quad}\CC\xrightarrow{1/\sqrt{2}}\CC
\end{equation}
of the $A_4$ zig-zag quiver, again with $\tau=\tr$, $\theta=0$.
The equations \eqref{gradflow} are
\begin{gather}
\dot{h_1}/h_1=\frac{1}{2}h_2/h_1, \qquad \dot{h_2}/h_2=-\frac{1}{2}h_2/h_1-h_2/h_3 \\
\dot{h_3}/h_3=\frac{1}{2}h_4/h_3+h_2/h_3, \qquad \dot{h_4}/h_4=-\frac{1}{2}h_4/h_3
\end{gather}
with asymptotic solution 
\begin{gather}
h_1=t^{1/2}(\log t)^{-1/2}\left(1+(\log t)^{-1}\right),\qquad h_2=t^{-1/2}(\log t)^{-1/2} \\
h_3=t^{1/2}(\log t)^{1/2},\qquad h_4=t^{-1/2}(\log t)^{1/2}\left(1+(\log t)^{-1}\right) 
\end{gather}
which is not an exact solution, but solves \eqref{gradflow} up to terms in $L^1$.
This is what we found in Section~\ref{sec_example} but with slightly different constants resulting from the change of basis on the representation.
\end{ex}

\bibliographystyle{plain}
\bibliography{balanced}

\Addresses

\end{document}